\documentclass[a4paper]{article}

\usepackage{amsmath}
\usepackage{amssymb}
\usepackage{amsthm}
\usepackage{mathrsfs}
\usepackage{wasysym}
\usepackage{bbm}		% blackboard 1
\usepackage{mathtools}
\usepackage[shortlabels]{enumitem}
\usepackage{braket}		% defines \middle similar to \left and \right
\usepackage{mdwlist}		% suspend/resume lists
\usepackage{xcolor} 
\usepackage{tikz}
\usepackage{graphicx}
\usepackage[a4paper, left=2.7cm, right=2.7cm, top=3.5cm, bottom=2cm]{geometry}
\usepackage[pdftex=true,hyperindex=true,pdfborder={0 0 0}]{hyperref}
\usetikzlibrary{patterns,positioning,arrows,decorations.markings,calc,decorations.pathmorphing,decorations.pathreplacing}

\theoremstyle{plain}
\newtheorem{theorem}{Theorem}[section]
\newtheorem{lemma}[theorem]{Lemma}
\newtheorem{corollary}[theorem]{Corollary}
\newtheorem{proposition}[theorem]{Proposition}
\newtheorem{definition}[theorem]{Definition}

\newtheorem{question}[theorem]{Question}
\newtheorem{claim}[theorem]{Claim}
\theoremstyle{definition}
\newtheorem{remark}[theorem]{Remark}
\newtheorem{example}[theorem]{Example}

\newcommand{\cuadri}{\square^{\mathbb{Z}^2}}
\newcommand{\GW}{X_{\texttt{GW}}}
\newcommand{\ZZ}{\mathbb{Z}}			% The set of integers
\newcommand{\NN}{\mathbb{N}}			% The set of natural numbers
\newcommand{\cubito}{\begin{tikzpicture}[scale = 0.2]
	\fill[black!30] (0,0) -- (1,0) -- (1.5,0.5) -- (1.5,1.5) -- (0.5,1.5) -- (0,1) -- (0,0);
	\draw (0,0) rectangle (1,1);
	\draw (0.5,1.5) -- (1.5,1.5);
	\draw (1.5,0.5) -- (1.5,1.5);
	\draw (1,0) -- (1.5,0.5);
	\draw (0,1) -- (0.5,1.5);
	\draw (1,1) -- (1.5,1.5);
	\end{tikzpicture}}
\newcommand{\nubito}{\begin{tikzpicture}[scale = 0.2]
	\draw (0,0) rectangle (1,1);
	\draw (0.5,0.5) rectangle (1.5,1.5);
	\draw (0,0) -- (0.5,0.5);
	\draw (1,0) -- (1.5,0.5);
	\draw (0,1) -- (0.5,1.5);
	\draw (1,1) -- (1.5,1.5);
	\end{tikzpicture}}

\newcommand{\IE}{\mathtt{IE}}
\newcommand{\R}{\mathtt{A}}
\newcommand{\EP}{\mathtt{E}}

\renewcommand{\complement}{%				% Set complement
	\mathsf{c}%
}
\newcommand{\supp}{%					% Support
	\operatorname{\mathrm{supp}}%
}

\title{%
	A hierarchy of topological systems with completely
	positive entropy\\
	\footnotetext{Last update:~\today}
}

\author{Sebasti\'an Barbieri\thanks{University of British Columbia}~ and Felipe Garc\'ia-Ramos\thanks{CONACyT \& Universidad Aut\'onoma de San Luis Potos\'i}}

\date{}

\begin{document}
	
	\maketitle
	
	%\vspace{-4em}
	%
	%\renewcommand{\contentsname}{}
	%{\footnotesize\tableofcontents}
	
	\begin{abstract}
	We define a hierarchy of systems with topological completely positive entropy in the context of continuous countable amenable group actions on compact metric spaces. For each countable ordinal we construct a dynamical system on the corresponding level of the aforementioned hierarchy and provide subshifts of finite type for the first three levels. We give necessary and sufficient conditions for entropy pairs by means of the asymptotic relation on systems with the pseudo-orbit tracing property, and thus create a bridge between a result by Pavlov and a result by Meyerovitch. As a corollary, we answer negatively an open question by Pavlov regarding necessary conditions for completely positive entropy. 	
%		A dynamical system has topological completely positive entropy (CPE) if every non-trivial factor has positive topological entropy. This property has been studied locally through entropy pairs, independence sets, asymptotic pairs and --in the case of subshifts-- pattern exchangeability. We construct a new bridge between these notions in the case of systems with the pseudo-orbit tracing property (POTP). Specifically, we relate the set of entropy pairs with the set of asymptotic pairs for which there exists an invariant measure supporting the pair. This motivates the definition of a hierarchy which orders systems with TCPE through two operations on the set of asymptotic pairs: topological closure and transitive closure. This formalism allows us to give a new proof of a result of Pavlov in the case of a countable amenable group. We also answer a question by Pavlov in the negative by producing a $\ZZ^3$-SFT which lies higher in the aforementioned hierarchy.
	\end{abstract}

\section{Introduction}

\bigskip

In ergodic theory a K-system is a measurable dynamical system that satisfies
Kolmogorov's zero-one law. These systems were characterized by Rokhlin and
Sinai as those where every non-trivial factor has positive entropy or as those
where every non-trivial partition has positive entropy
\cite{rohlin1961structure}. In \cite{blanchard1992fully} Blanchard introduced
two topological analogues of the K-systems, that is, topological dynamical
systems with uniformly positive entropy (UPE) and topological completely
positive entropy (CPE), which can be defined respectively as those for which
every standard open cover has positive entropy and those for which every non-trivial
factor has positive entropy. Even though for some families, such as $\mathbb{Z}$-subshifts of
finite type and expansive algebraic $\mathbb{Z}^{d}$-actions, UPE and
topological CPE coincide, in general, topological CPE\ does not imply UPE.

In order to understand the properties of these systems, Blanchard introduced
the notion of entropy pairs in \cite{Blanchard1993}. This seminal work is the
birth of what is now called local entropy theory (see \cite{glasner2009local}
for a survey). Loosely speaking, a pair of points is an entropy pair if every standard
open cover which separates them has positive topological entropy. Blanchard showed (for
$\mathbb{Z}$-actions) that a system has positive topological entropy if and
only if there exists an entropy pair; a system has UPE if and only if every
non-trivial pair is an entropy pair; and that the system has topological CPE if and only
if the smallest closed equivalence relation containing the entropy pairs is
the whole set $X^{2}$. These results were generalized to actions of countable amenable groups by Kerr and Li ~\cite{KerrLi2007} by means of characterizing entropy pairs with
the notion of independence. This point of view paved the way to a
combinatorial study of topological entropy; even in the case of sofic groups
actions. See Chapter 12 of~\cite{KerrLiBook2016}. 

In this paper we will introduce a hierarchy of topological dynamical systems
that lie between UPE and topological CPE (see Section 2). A noteworthy remark is that the set of entropy pairs plus the diagonal is closed but not
necessarily an equivalence relation. The hierarchy is defined as follows:\ the first level of the hierarchy corresponds to the systems where
the entropy pairs and the diagonal are the whole product space, that is, systems with UPE; the second
level consists of the systems which are not on the first level and such that the smallest equivalence relation that contains the
entropy pairs is the whole product space. The third level is constituted by the
systems which are neither on the first nor second level and such that the topological closure of the smallest equivalence relation that contains
the entropy pairs is the whole product space. Subsequent levels of the hierarchy correspond to those systems for which the smallest closed equivalence relation containing the entropy pairs is the whole product space, but that require a larger number of alternating transitive and topological closures to stabilize. We show that in general this hierarchy does not collapse at any countable ordinal.

{ \renewcommand{\thetheorem}{\ref{teo_onichan}} 
	
	\begin{theorem}
	For every countable ordinal $\alpha$ the CPE class $\alpha$ is non-empty. 
\end{theorem}
	
	\addtocounter{theorem}{-1}}

 Furthermore we show that in the first 3 levels of the
hierarchy we can find $\mathbb{Z}^{d}$-SFTs (Corollary~\ref{corolario_elronniedance} and Theorem~\ref{corolario_lapapa}). This exemplifies how rich, from the
topological dynamics point of view, is the class of SFTs. In order to construct
these examples we first needed to better understand entropy pairs in the setting of SFTs or
more generally on dynamical systems with the pseudo-orbit tracing property. It
turns out that they are closely related to the notion of asymptotic pairs.

We say $(x,y)$ is an asymptotic pair if for every $\varepsilon>0$ the set of
elements $g$ in the group for which $d(gx,gy)>\varepsilon$ is finite. In
\cite{ChungLi2014} Chung and Li asked if every expansive action of a countable amenable group with
positive topological entropy has a non-trivial (i.e. $x\neq y$) asymptotic pair.
Recently Meyerovitch~\cite{Meyerovitch2017} showed that in general this is
not true. Nonetheless there are two families: an algebraic one \cite{ChungLi2014}\cite{LindSchdmit1999}
and expansive systems with the pseudo-orbit tracing property (POTP) \cite{Meyerovitch2017} where the result holds.
% where the result
%holds, that is every expansive action of a countable amenable group with positive topological
%entropy has a non-trivial asymptotic pair.

In this paper, we characterize exactly how the local theory of entropy
and asymptotic pairs are related in the context of countable amenable group actions with the pseudo-orbit tracing property. We give a localized version of Meyerovitch's result by
using the formalism of entropy pairs and show that a sort of converse also holds.  
%
%A corollary of our result shows that in this setting, the
%asymptotic pairs are dense in the set of entropy pairs.

Let us state our result precisely. A $G$-topological dynamical system
($G$-TDS) is a pair $(X,T)$ where $X$ is a compact metric space and $T$ is a left
$G$-action on $X$ by homeomorphisms. Denote the diagonal of $X^{2}$ by
$\Delta$, the set of entropy pairs by $\mathtt{E}(X,T)$ and by $\mathtt{A}%
^{\varepsilon}(X,T)$ the set of pairs $(x,y)$ for which $d(gx,gy)>\varepsilon$
for finitely many $g\in G$ and by $\mathtt{A}(X,T)=\bigcap_{\varepsilon
>0}\mathtt{A}^{\varepsilon}(X,T)$ the set of asymptotic pairs. Let $G$ be a
countable amenable group. Our result can be stated as follows:

{ \renewcommand{\thetheorem}{\ref{pipeteorema}} 

\begin{theorem}
	Let $(X,T)$ be a $G$-TDS with the pseudo-orbit tracing property.
	
	\begin{enumerate}
		\item If $(x,y)\in\EP(X,T)$ then $(x,y)\in\overline{\R^{\varepsilon
			}(X,T)}\diagdown\Delta$ for every $\varepsilon>0$ and there exists an
		invariant measure $\mu$ such that $x,y\in \operatorname{supp}(\mu).$
		
		\item If $(x,y)\in \R(X,T)$ and there exists an invariant
		measure $\mu$ such that $x,y\in \operatorname{supp}(\mu)$ then
		$(x,y)\in\EP(X,T)\cup\Delta$. 
	\end{enumerate}
\end{theorem}

\addtocounter{theorem}{-1}}

Even though the hierarchy is an abstract construction this theorem provides a very concrete way to look at entropy pairs. For instance, it tells us that UPE systems with the pseudo-orbit tracing property and a fully supported measure have a dense set of asymptotic pairs. This theorem is also the main tool used in this paper to check if certain particular examples have UPE or CPE and to determine what precise class they belong to.

In the case where $(X,T)$ is expansive, we have that $\mathtt{A}(X,T) =
\mathtt{A}^{\varepsilon}(X,T)$ for some positive $\varepsilon$. In particular, Theorem~\ref{pipeteorema} allows us to recover Meyerovitch's result.

%This means
%that if $(X,T)$ is an expansive $G$-TDS with the POTP, then $(x,y)\in
%\mathtt{E}(X,T)$ if and only if $(x,y)\in$ $\overline{\mathtt{A}%
%	(X,T)}\diagdown\Delta\cap\operatorname{\mathrm{supp}}(\mu\times\mu)$ for some
%invariant measure $\mu.$ 

Other results that relate asymptotic pairs and positive topological entropy
are those of Pavlov. In~\cite{Pavlov2013} Pavlov characterized $\mathbb{Z}%
^{d}$-SFTs with a fully supported measure having only positive entropy
symbolic factors through a combinatorial condition on the patterns, called
chain exchangability (CE); loosely speaking, a subshift has chain
exchangability if every pattern can be turned into any other pattern over the
same support by constructing a chain of patterns such that any two consequent
patterns on the chain can be obtained by localizing each of the configurations
on an asymptotic pair. Subsequently, in~\cite{Pavlov2017}, he gave sufficient
conditions which imply topological CPE for $\mathbb{Z}^{d}$-SFTs; his
condition being bounded chain exchangeability (BCE). In the same paper, Pavlov
asked whether the condition of having BCE was necessary to have topological
CPE in the context of $\mathbb{Z}^{d}$-subshifts of finite type.

As another consequence of Theorem \ref{pipeteorema}\ we show that any subshift of finite type
satisfying Pavlov's BCE condition (with a fully supported measure) must belong
to the first or second level of our hierarchy of complete entropy. In
particular, this gives a new proof of Pavlov's result and extends it to
countable amenable groups. Furthermore, we answer his question in the negative by constructing an SFT that we denote as the Good Wave Shift.

{ \renewcommand{\thetheorem}{\ref{teoremabonito}} 
	
	\begin{theorem}
		There is a topologically weakly mixing $\mathbb{Z}^{3}$-SFT with topological CPE which does not have BCE.
	\end{theorem}
	
	\addtocounter{theorem}{-1}}

\section{Background}

Let $G$ be a group. We denote by $F \Subset G$ a finite subset of $G$. A sequence $\{F_n\}_{n \in \NN}$ of finite subsets of $G$ is said to be (left) \textbf{asymptotically invariant} or \textbf{F\o lner} if for every $K \Subset G$ we have that $|K \cdot F_n \Delta F_n|/|F_n| \to 0$. A countable group is \textbf{amenable} if it admits a F\o lner sequence. 

From now on, $G$ denotes an arbitrary countable amenable
group. A \textbf{$G$-topological dynamical system ($G$-TDS)} is a pair $(X,T)$ where
$X$ is a compact metric space and $T : G \times X \to X$ is a
left $G$-action on $X$ by homeomorphisms $T(g,x) = T^g(x)$. We say that a system $(Y,S)$ is a \textbf{factor} of $(X,T)$ if there exists a continuous surjective $G$-equivariant map $\phi: X \to Y$. For an open cover $\mathcal{U}$ of $X$, and $F \Subset G$ we denote by $\mathcal{U}^{F}=\bigvee\nolimits_{g\in F}%
T^{g^{-1}}\mathcal{U}$ the refinement of $\mathcal{U}$ by $F$. We also denote the minimum cardinality of a subcover of $\mathcal{U}$ by
$N(\mathcal{U)}$.

\subsection{Entropy pairs and the CPE class}

\begin{definition}
Let $(X,T)$ be a $G$-TDS, $\mathcal{U}$ an open cover and $\left\{
F_{n}\right\}_{n \in \NN}$ a F\o lner sequence. We define the \textbf{topological entropy of
$(X,T)$ with respect to $\mathcal{U}$} as%
\[
h_{\text{top}}(T,\mathcal{U})=\lim_{n\rightarrow\infty}\frac{1}{\left\vert
F_{n}\right\vert }\log N(\mathcal{U}^{F_{n}}).
\]

\end{definition}

Note that this limit does not depend on the choice of F\o lner sequence, see for instance Theorem 4.38 in~\cite{KerrLiBook2016}. The \textbf{topological entropy} of $(X,T)$ is defined as \[
h_{\text{top}}(T)=\sup_{\mathcal{U}}h_{\text{top}}(T,\mathcal{U}).
\]

\begin{definition}
We say a $G$-TDS $(X,T)$ has \textbf{uniform positive entropy
(UPE)} if for each standard (two non-dense open sets) cover $\mathcal{U}$ we have that
$h_{\text{top}}(T,\mathcal{U})>0.$ We say it has\textbf{ topological
completely positive entropy (topological CPE)} if each non-trivial factor has
positive topological entropy.\end{definition}

In general we have that every system with UPE has topological CPE but the converse is not true even if the system is minimal \cite{song2009minimal}.

\begin{definition}
Let $(X,T)$ be a $G$-TDS. A pair $(x,y)\in X^{2}$ is an \textbf{entropy pair} if $x\neq
y$ and for every pair of disjoint closed neighborhoods $U_{x},U_{y}$ of $x$ and $y$
respectively we have that
\[
h_{\text{top}}(T,\left\{  U_{x}^{\complement},U_{y}^{\complement}\right\}  )>0.
\]
The set of entropy pairs is denoted by $\mathtt{E}(X,T).$
\end{definition}

\begin{theorem}\label{teoremadelblansharr}
[\cite{Blanchard1993}\cite{KerrLi2007}] Let $(X,T)$ be a $G$-TDS and $\Delta = \{(x,x) \mid x \in X\}$ the diagonal of $X^2$.
\begin{enumerate}
	\item There exists an entropy pair if and only if $h_{\text{top}}(T)>0$.
	\item $\mathtt{E}(X,T) \cup \Delta=X^{2}$ if and only if $(X,T)$ has UPE.
	\item The smallest closed equivalence relation that contains the entropy
	pairs is $X^{2}$ if and only if $(X,T)$ has topological CPE.
\end{enumerate}
\end{theorem}

\bigskip Given a subset $R\subset X^{2}$ we denote with $R^{+}$ the smallest
equivalence relation that contains $R$. Considering that the union of the set of entropy pairs and the diagonal is closed but not an equivalence
relation, we define the following sets inductively (transfinitely). Let
$\alpha$ be an ordinal. We define%
\[
\mathtt{E}_{1}(X,T):=\mathtt{E}(X,T)\cup\Delta\text{, and}%
\]
\begin{align*}
\mbox{ If $\alpha$ is a successor } &  \ \ \ \mathtt{E}_{\alpha}(X,T):=%
\begin{cases}
\overline{\mathtt{E}_{\alpha-1}(X,T)} & \mbox{ if }\mathtt{E}_{\alpha
-1}(X,T)\mbox{ is not closed,}\\
\mathtt{E}_{\alpha-1}(X,T)^{+} & \mbox{ otherwise.}
\end{cases}
\\
\mbox{ If $\alpha$ is a limit } &  \ \ \ \mathtt{E}_{\alpha}(X,T):=\bigcup
_{\beta<\alpha}\mathtt{E}_{\beta}(X,T)
\end{align*}

\bigskip The following definition introduces a hierarchy of systems that lie between UPE and
topological CPE.

\begin{definition}
Let $\alpha$ be an ordinal. We say that a dynamical system $(X,T)$ is in the
\textbf{CPE class $\alpha$} if $E_{\alpha}(X,T)=X^{2}$ and for
every $\beta<\alpha$ we have $E_{\beta}(X,T)\neq X^{2}$.
\end{definition}

Note that $(X,T)$ is in the CPE class $1$ if and only if $(X,T)$ has UPE. 

\begin{proposition}\label{prop_onichan}
	A $G$-TDS $(X,T)$ is in the CPE class $\alpha$ for some ordinal $\alpha$ if and only if it has topological CPE.
\end{proposition}

\begin{proof}
	If $(X,T)$ is in the CPE class $\alpha$ then the smallest closed equivalence relation containing $E(X,T)$ is $X^2$ therefore by Theorem~\ref{teoremadelblansharr} it has topological CPE. Conversely, by the same result, the smallest closed equivalence relation containing $E(X,T)$ is $X^2$. As the increasing chain $\{\overline{E_{\alpha}(X,T)}\}_{\alpha}$ of closed sets is contained in $X^2$ which is separable, it must stabilize at a countable ordinal. See for instance Chapter 1, exercise 18 of~\cite{akin2010general}.
\end{proof}

As far as the authors are aware, every example in the literature of a $G$-TDS which has topological CPE but not UPE is either in the CPE class $2$ or $3$. For instance, Blanchard's example from~\cite{blanchard1992fully} $X = \{a,b\}^{\ZZ} \cup \{a,c\}^{\ZZ}$ with the shift action clearly is in the CPE class $2$. And Song and Ye's minimal example from~\cite{song2009minimal} is diagonal, that is, $(x,Tx) \in E(X,T)$ for every $x \in X$. Therefore it satisfies that $\overline{E(X,T)^{+}} = X^2$, meaning that it cannot belong to any class above $3$. We will later prove that for every countable ordinal the corresponding CPE class is non-empty. Furthermore, we will also show the existence of subshifts of finite type in the first three classes.

\subsection{Asymptotic pairs and the asymptotic class}

Here we define a similar class but based on asymptotic pairs. The main motivation for introducing this class is that under some conditions given in Section~\ref{section_teofel}, the asymptotic class and the CPE class coincide. Nonetheless, since asymptotic pairs on dynamical systems are natural objects which can be used to study chaotic behaviour (see for references~\cite{BLANCHARD2002,Downarowicz2011,Huang2015}), we believe the asymptotic class is interesting on its own.

\begin{definition}
\label{potp}\bigskip Let $(X,T)$ a $G$-TDS. We say $(x,y)$ is an \textbf{$\varepsilon
$-asymptotic pair} if there exists $F \Subset G$ such that for $g \notin F$ $d(T^{g}x,T^{g}y)\leq\varepsilon$. Furthermore, we say $(x,y)$ is an \textbf{asymptotic pair} if it is $\varepsilon$-asymptotic for
every $\varepsilon>0$.
\end{definition}

We denote the $\varepsilon$-asymptotic pairs with $\R^{\varepsilon}(X,T)$ and
the asymptotic pairs with $\R(X,T).$ The asymptotic pairs form an equivalence
relation that in general is not closed. Let $(X,T)$ be a $G$-TDS. For an ordinal
$\alpha$ we define the following increasing set of asymptotic relations.
\[
\R_{0}(X,T):=\R(X,T)\text{, and}%
\]
\begin{align*}
\mbox{ If $\alpha$ is a successor }  &  \ \ \ \R_{\alpha}(X,T) :=
\begin{cases}
\overline{\R_{\alpha-1}(X,T)} & \mbox{ if } \R_{\alpha-1}(X,T)
\mbox{ is not closed,}\\
\R_{\alpha-1}(X,T)^{+} & \mbox{ otherwise.}\\
\end{cases}
\\
\mbox{ If $\alpha$ is a limit }  &  \ \ \ \R_{\alpha}(X,T) := \bigcup_{\beta<
\alpha}\R_{\beta}(X,T)
\end{align*}

%In the case of $n \in\mathbb{N}$ we have that $\R_{n}(X,T) = \overline
%{\R_{n-1}(X,T)}$ if $n$ is odd and $\R_{n}(X,T) = {\R_{n-1}(X,T)}^{+}$ if $n$ is even.

For an ordinal $\alpha$ we say that a $G$-TDS $(X,T)$ \textbf{is in the asymptotic class
$\alpha$} if $\R_{\alpha}(X,T)=X^{2}$ and for every $\beta<\alpha$ we have
$\R_{\beta}(X,T)\neq X^{2}$. We remark that every system for which the smallest closed equivalence relation containing $\R(X,T)$ is $X^2$ must satisfy that $\R_{\alpha}(X,T) = X^2$ for some countable ordinal $\alpha$ by analogous reasons as those given in Proposition~\ref{prop_onichan}.

\subsection{Independence}

The notion of independence can be used to characterize entropy pairs. A suggested reference for this topic is Chapter 12 of Kerr and Li's book~\cite{KerrLiBook2016}.

\begin{definition}
We say $J\subset G$ has \textbf{positive density $D(J)$ with respect to a F\o lner
sequence} $\left\{  F_{n}\right\}_{n \in \NN}$ if the following limit exists and is positive
\[D(J) := 
\lim_{n\rightarrow\infty}\frac{\left\vert F_{n}\cap J\right\vert }{\left\vert
F_{n}\right\vert }>0.
\]

\end{definition}

\begin{definition}
Let $(X,T)$ be a $G$-TDS. Let $\mathbf{A}=(A_{1},\dots,A_{n})$ be a tuple of
subsets of $X.$ We say $J\subset G$ is an \textbf{independence set}\textit{
for} $\mathbf{A}$ if for every nonempty $I\Subset J$ and any
$\phi : I \to \{1,\dots,n\}$ we have
\[
\bigcap_{i\in I}T^{i^{-1}}A_{\phi(i)}\neq\emptyset
\]

\end{definition}

\begin{definition}
We say $(x,y)$ is an \textbf{independence entropy pair (IE-pair)} if every
pair of open sets $\mathbf{U} = \{U_{1},U_{2}\}$ with $x_{1}\in U_{1}$ and $x_{2}\in U_{2},$
admits an independence set $J$ with positive density for some F\o lner sequence. We denote the set of independence entropy pairs with $\IE(X,T).$
\end{definition}

\begin{definition}
Let $(X,T)$ be a $G$-TDS and $\mu$ an invariant measure. We say a F\o lner
sequence $\left\{  F_{n}\right\}_{n \in \NN}  $ \textbf{satisfies the pointwise ergodic
theorem (PET)} if for every $f \in L^1(X)$ we have that
\[
\lim_{n\rightarrow\infty}\frac{1}{\left\vert F_{n}\right\vert }\sum_{g\in
F_{n}}f \circ T^{g}x=\mathbb{E}(f)(x). \ \ \ \ \mu\mbox{-a.e.}
\]
Where $\mathbb{E}$ is the conditional expectation of $\mu$ with respect to the subspace of $G$-invariant functions in $L^1(X)$.
\end{definition}

\begin{remark}
	Every countable amenable group admits a F\o lner sequence which satisfies the
	PET, see~\cite{shulman1988maximal,Lindenstrauss2001}.
\end{remark}

\begin{remark}
[Proposition 12.7 \cite{KerrLiBook2016}]Let $(x,y)$ be an IE-pair. For
every F\o lner sequence $\left\{  F_{n}\right\}_{n \in \NN}$ that satisfies the PET there exists an independence
set with positive density with respect to $\left\{  F_{n}\right\}_{n \in \NN}$.
\end{remark}

Let $\mathcal{S}$ be the set of all $x \in X$ for which there exists an invariant measure $\mu$ for which $x \in \operatorname{supp}(\mu)$ and $\Delta _{\mathcal{S}}=\left\{ (x,x)\in X^{2}: x \in \mathcal{S} \right\}$ be the diagonal of $\mathcal{S}$.
\begin{theorem}\label{teorema_ie_es_ep}
[\cite{KerrLi2007} or Theorem 12.20 \cite{KerrLiBook2016}]\label{kl} Let
$(X,T)$ be a $G$-TDS. Then $\IE(X,T)=\EP(X,T)\cup\Delta _{\mathcal{S}}.$
\end{theorem}

\subsection{Pseudo-orbit tracing property}

\begin{definition}
Let $(X,T)$ a $G$-TDS, $\delta>0$ and $S\Subset G$. An $(S,\delta)$
\textbf{pseudo-orbit} is a sequence $(x_{g})_{g\in G}$ such that $d(T^{s}%
x_{g},x_{s\cdot g})<\delta$ for all $s\in S$ and $g\in G.$ We say a
pseudo-orbit is $\mathbf{\varepsilon}$-\textbf{traced} \textbf{by }$x$ if
$d(T^{g}x,x_{g})\leq\varepsilon$ for all $g\in G.$
\end{definition}

\begin{definition}
A $G$-TDS $(X,T)$ has the \textbf{pseudo-orbit tracing property (POTP)} if for
every $\varepsilon>0$ there exists $\delta>0$ and a finite set $S$ such that
any $(S,\delta)$ pseudo-orbit is $\varepsilon$-traced by some point $x\in X.$
\end{definition}

This condition is also known as shadowing. In the case where the space $X$ is zero-dimensional, any expansive $G$-TDS
$(X,T)$ is topologically conjugate to a subshift. In this context, the systems
satisfying the POTP are exactly the subshifts of finite type (SFTs). That is,
they can be seen as sets of coloring of the group $G$ by a finite alphabet
which can be characterized by forbidding the appearance of a finite set of
patterns. This equivalence was first shown by Walters~\cite{Walters1978} in
the context of $\mathbb{Z}$-actions and subsequently generalized to
$\mathbb{Z}^{d}$ by Oprocha~\cite{Oprocha2008} and arbitrary finitely
generated groups by Chung and Lee~\cite{ChungLi2017}. This last result also
covers the arbitrary countable case, as remarked by
Meyerovitch~\cite{Meyerovitch2017}.

Other families of systems satisfying the POTP are axiom A diffeomorphisms~\cite{Bowen1978} and principal algebraic actions of countable amenable groups~\cite{Meyerovitch2017}.

\section{Examples for every CPE class}

\begin{lemma}
	\label{lemadelencaje} Let $\alpha$ be an ordinal and $(X,T)$,$(Y,S)$ two $G$-TDS. If $\phi:X \longrightarrow Y$ is a $G$-equivariant continuous map then
	$\phi(\EP_{\alpha}(X,T))\subset\EP_{\alpha}(Y,S)$.
\end{lemma}

\begin{proof}
	We proceed using transfinite induction. The base step is already
	known, see~\cite{Blanchard1993}. However, we will do the proof for completeness. Let $(x,y)\in\EP_{1}(X,T)$. If $\phi(x)=\phi(y)$ then clearly $(\phi
	(x),\phi(y))\in\EP_{1}(Y,S)$. Assume $\phi(x)\neq\phi(y)$ and let $V_{\phi(x)},V_{\phi(y)}$ be disjoint closed
	neighborhoods of $\phi(x)$ and $\phi(y)$ respectively. Since $(x,y)\in
	\EP_{1}(X,T)$ we have that
	\[
	h_{top}\left(T,\left\{  \left(  \phi^{-1}V_{\phi(x)}\right)  ^{c},\left(  \phi
	^{-1}V_{\phi(x)}\right)  ^{c}\right\}\right) >0.
	\]
	
	By continuity and $G$-equivariance of $\phi$ we have that
	 $$h_{top}\left(S,\left\{  \left(  V_{\phi(x)}\right)
	^{c},\left(  V_{\phi(x)}\right)  ^{c}\right\}\right)  =h_{top}\left(T,\left\{  \left(  \phi^{-1}V_{\phi(x)}\right)  ^{c},\left(  \phi
	^{-1}V_{\phi(x)}\right)  ^{c}\right\}\right).$$
	Therefore $(\phi(x),\phi(y)) \in \EP_{1}(Y,S)$. Now let $\alpha$ be an ordinal and assume that $\phi(\EP_{\beta}(X,T))\subset\EP_{\beta}(Y,S)$
	for every $\beta<\alpha$.
	
	It is not hard to see that if $\phi(A)\subset B$ then $\phi(\overline
	{A})\subset \overline{B}$ and $\phi(A^{+}) \subset B^{+}$. This proves the result for any successor ordinal $\alpha$. In the case of a limit ordinal we have that $\phi(A_{\beta})\subset B_{\beta}$ and therefore $\phi(A_{\alpha}) = \phi(\cup_{\beta < \alpha} A_{\beta})\subset \cup_{\beta < \alpha} B_{\beta} = B_{\alpha}.$ Hence the result also holds for $\alpha$ limit.
\end{proof}	
%	$\EP_{\alpha}(X,T)$ is not closed. \ Since $\phi$ is continous $\phi
%	(\overline{\EP_{\alpha-1}(X,T)})\subset\overline{\EP_{\alpha-1}(Y,S)}.$

\begin{remark}\label{remark_delencaje}
	If we consider the asymptotic class in Lemma~\ref{lemadelencaje} instead of the CPE class the result still holds. The image under a $G$-equivariant continuous map of an asymptotic pair is clearly asymptotic, and the rest of the argument is the same.
\end{remark}

\bigskip

We will define a family of dynamical systems $(X_{\alpha},T_{\alpha})$ for every countable ordinal $\alpha$ and a pair of fixed points
$x_{\alpha},y_{\alpha}$ whenever $\alpha$ is a successor. Let $X_{1}:=\{0,1\}^{G}$, $T_{1}:=\sigma: G \times X\longrightarrow X$ denote
the left shift $G$-action on $X$ where for every $g \in G$ we have $\sigma^g(x)_h = x_{g^{-1}h}$ for every $h \in G$, and $x_{1}=0^{G}$, $y_{1}=1^{G}$ be fixed points for $\sigma$. Let $I:=\{\frac{1}{n}~:~n > 0\}\cup\{0\}$ with the Euclidean topology. 

%, and neighborhoods $U_{x_{\alpha}},U_{y_{\alpha}}.$

\begin{enumerate}%caso1
	\item If $\alpha$ is an even successor ordinal we define
	\[
	X_{\alpha}:=X_{\alpha-1}\times\left\{1,2,3 \right\}  \diagup\left[
	(y_{\alpha-1},1) \sim (y_{\alpha-1},2) \wedge (x_{\alpha-1},2) \sim (x_{\alpha-1},3) \right]  ,
	\]
	$$x_{\alpha}  :=(x_{\alpha-1},1),\text{ }y_{\alpha} :=(y_{\alpha-1},3)$$
%	$$U_{x_{\alpha}} = U_{x_{\alpha-1}}\times\left\{  1\right\},\text{ }U_{y_{\alpha}} = U_{y_{\alpha-1}}\times\left\{  3\right\} ,$$
	with $\tau_{\alpha}$ the product topology and $T_{\alpha}^{g}:X_{\alpha}\longrightarrow X_{\alpha}$ the map
	\[
	T_{\alpha}^{g}((x,i)):=(T_{\alpha-1}^{g}(x),i)\text{ } \mbox{ for every } x\in X_{\alpha-1}\mbox{
		and }i \in \{1,2,3\}.
	\]
	
	%caso 2
	\item If $\alpha$ is an odd successor ordinal and $\alpha-1$ is not a limit
	ordinal we define
	\[
	X_{\alpha}:=X_{\alpha-1}\times\left\{  I\right\}  \diagup\left[ \begin{split}
	(y_{\alpha-1},\tfrac{1}{n})\sim(y_{\alpha-1},\tfrac{1}{n+1}), & ~\forall n \mbox{ odd}\\
	(x_{\alpha-1},\tfrac{1}{n})\sim(x_{\alpha-1},\tfrac{1}{n+1}), & ~\forall n > 0 \mbox{ even}
	\end{split}\right],
	\]
	$$
	x_{\alpha} :=(x_{\alpha-1},1),\text{ }y_{\alpha} :=(y_{\alpha-1},0),$$
%	$$U_{x_{\alpha}} = U_{x_{\alpha-1}}\times\left\{  1\right\},\text{ }U_{y_{\alpha}} = U_{y_{\alpha-1}}\times\left\{  0\right\} ,$$
	with $\tau_{\alpha}$ the product topology and $T_{\alpha}^{g}:X_{\alpha}\longrightarrow X_{\alpha}$ the map
\[
T_{\alpha}^{g}((x,i)):=(T_{\alpha-1}^{g}(x),i)\text{ } \mbox{ for every } x\in X_{\alpha-1}\mbox{
	and }i \in I.
\]
%caso 3
	\item If $\alpha$ is a countable limit ordinal, fix an increasing sequence of successor ordinals $\left\{
	\alpha(n)\right\}_{n \in \NN}$ such that $\alpha(n)\longrightarrow\alpha$ and define
	\[
	X_{\alpha}:=\bigsqcup_{n \in \NN}X_{\alpha(n)}\diagup\left[  y_{\alpha
		(n)}\sim y_{\alpha(n+1)},~\forall n\in \NN\right].
	\]
	Here $\bigsqcup$ stands for disjoint union. We consider the topology $\tau_{\alpha} = \tau^1_{\alpha} \cup \tau^2_{\alpha}$ where
	\begin{align*}
	\tau^1_{\alpha} := &  \{U~: \mbox{ there is } n \in \NN \mbox{ such that } U\in\tau_{\alpha(n)} \mbox{ and } y_{\alpha(n)} \notin U\}\\
	\tau^2_{\alpha} := &   \{V~: \mbox{ there are } n,m \in \NN \mbox{ such that } V=(X_{\alpha}\setminus\bigsqcup_{k< n}X_{\alpha(k)})\cup U \mbox{ and }y_{\alpha(m)}\in U\in\tau_{\alpha(m)}\}.
	\end{align*}
	Here $T_{\alpha}^{g}(x) = T_{\alpha(n)}^{g}(x)$ if $x \in X_{\alpha(n)}$.

	%caso4
	\item If $\alpha-1$ is a limit ordinal, let $\left\{
	\alpha(n)\right\}_{n \in \NN}$ be the sequence used to define $X_{\alpha-1}$. We define
	\[
	X_{\alpha}:=X_{\alpha-1}\times\left\{  I\right\}  \diagup\left[
	(x_{\alpha(n)},\tfrac{1}{n})\sim(x_{\alpha(n)},\tfrac{1}{n+1}),~\forall n>0\right]  ,
	\]
	$$
	x_{\alpha} := (x_{\alpha(0)},1),\text{ }y_{\alpha} := (x_{\alpha(0)},0)$$
%	$$U_{x_{\alpha}} := U_{x_{\alpha(0)}} \times\left\{  1\right\},\text{ }U_{y_{\alpha}} = U_{x_{\alpha(0)}}\times\left\{  0\right\} ,$$
	with $\tau_{\alpha}$ the product topology and $T_{\alpha}^{g}:X_{\alpha}\longrightarrow X_{\alpha}$ the map
	\[
	T_{\alpha}^{g}((x,i)):=(T_{\alpha-1}^{g}(x),i)\text{ } \mbox{ for every } x\in X_{\alpha-1}\mbox{
		and }i \in I.
	\]
	%	\[
	%	X_{\alpha+1}:=\sqcup_{\alpha^{\prime}<\alpha}X_{\alpha^{\prime}}\diagup\left[
	%	x_{\alpha^{\prime}}\sim x_{\alpha^{\prime}+1}\right]  ,
	%	\]
	
\end{enumerate}

One can check that on every step $X_{\alpha}$ is compact, Hausdorff and second countable
with the topology $\tau_{\alpha},$ and hence metrizable. Also it is not hard
to see that $T_{\alpha}$ is a continuous left $G$-action. We schematize the four cases of the construction in Figure~\ref{figure_lachupadelmate}

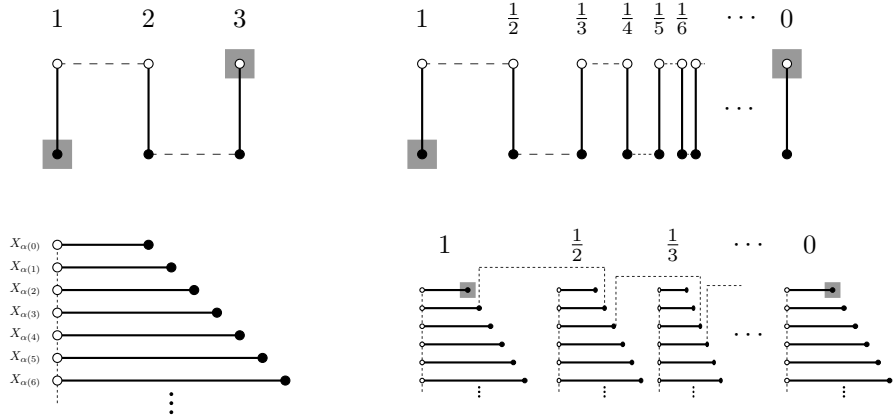
\begin{figure}[h!]
	\centering
	\begin{tikzpicture}[scale=0.6]

\begin{scope}[shift = {(0,0)}]
\node at (0,3) {${\tiny 1}$};
\node at (2,3) {${\tiny 2}$};
\node at (4,3) {${\tiny 3}$};
\draw[thick] (0,0) -- (0,2);
\draw[thick] (2,0) -- (2,2);
\draw[thick] (4,0) -- (4,2);
\draw[dashed, opacity = 0.8] (0,2) -- (2,2);
\draw[dashed, opacity = 0.8] (2,0) -- (4,0);
\draw[thick, fill = black, opacity = 0.4] (-0.3,-0.3) rectangle (0.3,0.3);
\draw[thick, fill = black, opacity = 0.4] (4-0.3,2-0.3) rectangle (4+0.3,2+0.3);
\draw[fill = black] (0,0) circle (0.1);
\draw[fill = white] (0,2) circle (0.1);
\draw[fill = black] (2,0) circle (0.1);
\draw[fill = white] (2,2) circle (0.1);
\draw[fill = black] (4,0) circle (0.1);
\draw[fill = white] (4,2) circle (0.1);
\end{scope}

\begin{scope}[shift = {(8,0)}]
\node at (0,3) {${\tiny 1}$};
\node at (2,3) {${\tiny \tfrac{1}{2}}$};
\node at (3.5,3) {${\tiny \tfrac{1}{3}}$};
\node at (4.5,3) {${\tiny \tfrac{1}{4}}$};
\node at (5.2,3) {${\tiny \tfrac{1}{5}}$};
\node at (5.7,3) {${\tiny \tfrac{1}{6}}$};
\node at (7,3) {${\tiny \cdots}$};
\node at (8,3) {${\tiny 0}$};
\draw[thick] (0,0) -- (0,2);
\draw[thick] (2,0) -- (2,2);
\draw[thick] (3.5,0) -- (3.5,2);
\draw[thick] (4.5,0) -- (4.5,2);
\draw[thick] (5.2,0) -- (5.2,2);
\draw[thick] (5.7,0) -- (5.7,2);
\draw[thick] (6,0) -- (6,2);
\draw[thick] (8,0) -- (8,2);
\node at (7,1) {$\cdots$};
\draw[dashed, opacity = 0.8] (0,2) -- (2,2);
\draw[dashed, opacity = 0.8] (2,0) -- (3.5,0);
\draw[dash pattern=on 2pt off 2pt, opacity = 0.8] (3.5,2) -- (4.5,2);
\draw[dash pattern=on 1pt off 1pt, opacity = 0.8] (4.5,0) -- (5.2,0);
\draw[dash pattern=on 1pt off 1pt, opacity = 0.8] (5.2,2) -- (5.7,2);
\draw[dash pattern=on 0.5pt off 0.5pt, opacity = 0.8] (5.7,0) -- (6,0);
\draw[dash pattern=on 0.5pt off 0.5pt, opacity = 0.8] (6,2) -- (6.2,2);
\draw[thick, fill = black, opacity = 0.4] (-0.3,-0.3) rectangle (0.3,0.3);
\draw[thick, fill = black, opacity = 0.4] (8-0.3,2-0.3) rectangle (8+0.3,2+0.3);
\draw[fill = black] (0,0) circle (0.1);
\draw[fill = white] (0,2) circle (0.1);
\draw[fill = black] (2,0) circle (0.1);
\draw[fill = white] (2,2) circle (0.1);
\draw[fill = black] (3.5,0) circle (0.1);
\draw[fill = white] (3.5,2) circle (0.1);
\draw[fill = black] (4.5,0) circle (0.1);
\draw[fill = white] (4.5,2) circle (0.1);
\draw[fill = black] (5.2,0) circle (0.1);
\draw[fill = white] (5.2,2) circle (0.1);
\draw[fill = black] (5.7,0) circle (0.1);
\draw[fill = white] (5.7,2) circle (0.1);
\draw[fill = black] (6,0) circle (0.1);
\draw[fill = white] (6,2) circle (0.1);
\draw[fill = black] (8,0) circle (0.1);
\draw[fill = white] (8,2) circle (0.1);
\end{scope}

\begin{scope}[shift = {(0,-5)}]
\begin{scope}[shift = {(-0.7,3)}, scale=0.8,transform shape]
\node at (0,0) {$X_{\alpha(0)}$};
\end{scope}
\begin{scope}[shift = {(-0.7,2.5)}, scale=0.8,transform shape]
\node at (0,0) {$X_{\alpha(1)}$};
\end{scope}
\begin{scope}[shift = {(-0.7,2)}, scale=0.8,transform shape]
\node at (0,0) {$X_{\alpha(2)}$};
\end{scope}
\begin{scope}[shift = {(-0.7,1.5)}, scale=0.8,transform shape]
\node at (0,0) {$X_{\alpha(3)}$};
\end{scope}
\begin{scope}[shift = {(-0.7,1)}, scale=0.8,transform shape]
\node at (0,0) {$X_{\alpha(4)}$};
\end{scope}
\begin{scope}[shift = {(-0.7,0.5)}, scale=0.8,transform shape]
\node at (0,0) {$X_{\alpha(5)}$};
\end{scope}
\begin{scope}[shift = {(-0.7,0)}, scale=0.8,transform shape]
\node at (0,0) {$X_{\alpha(6)}$};
\end{scope}

\draw[thick] (0,3) -- (2,3);
\draw[thick] (0,2.5) -- (2.5,2.5);
\draw[thick] (0,2) -- (3,2);
\draw[thick] (0,1.5) -- (3.5,1.5);
\draw[thick] (0,1) -- (4,1);
\draw[thick] (0,0.5) -- (4.5,0.5);
\draw[thick] (0,0) -- (5,0);
\draw[dash pattern=on 1pt off 1pt, opacity = 0.8] (0,-0.5) -- (0,3);
\foreach \x in {0,...,6}{
	\draw[fill = white] (0,0.5*\x) circle (0.1);
	\draw[fill = black] (5-0.5*\x,0.5*\x) circle (0.1);
	}
	\draw[fill = black] (2.5,-0.3) circle (0.03);
	\draw[fill = black] (2.5,-0.5) circle (0.03);
	\draw[fill = black] (2.5,-0.7) circle (0.03);

\end{scope}

\def \limiteeeeeeeeeeeee{
	
	\draw[thick] (0,4) -- (2,4);
	\draw[thick] (0,3.2) -- (2.5,3.2);
	\draw[thick] (0,2.4) -- (3,2.4);
	\draw[thick] (0,1.6) -- (3.5,1.6);
	\draw[thick] (0,0.8) -- (4,0.8);
	\draw[thick] (0,0) -- (4.5,0);
	\draw[dash pattern=on 1pt off 1pt, opacity = 0.8] (0,-0.5) -- (0,4);
	\foreach \x in {0,...,5}{
		\draw[fill = white] (0,0.8*\x) circle (0.1);
		\draw[fill = black] (4.5-0.5*\x,0.8*\x) circle (0.1);
	}
	\draw[fill = black] (2.5,-0.3) circle (0.03);
	\draw[fill = black] (2.5,-0.5) circle (0.03);
	\draw[fill = black] (2.5,-0.7) circle (0.03);	
}

\begin{scope}[shift = {(8,-5)}]
\begin{scope}[shift = {(0,0)}, scale = 0.5]
\limiteeeeeeeeeeeee 
\node at (1,6) {${\tiny 1}$};
\draw[thick, fill = black, opacity = 0.4] (2-0.3,4-0.3) rectangle (2+0.3,4+0.3);
\end{scope}
\begin{scope}[shift = {(3,0)}, xscale = 0.4, yscale = 0.5]
\limiteeeeeeeeeeeee
\node at (1,6) {${\tiny \tfrac{1}{2}}$};
\end{scope}
\draw[dash pattern=on 1pt off 1pt, opacity = 0.8] (1.25,1.6) -- (1.25,2.5);
\draw[dash pattern=on 1pt off 1pt, opacity = 0.8] (4,1.6) -- (4,2.5);
\draw[dash pattern=on 1pt off 1pt, opacity = 0.8] (1.25,2.5) -- (4,2.5);
\begin{scope}[shift = {(5.2,0)}, xscale = 0.3, yscale = 0.5]
\limiteeeeeeeeeeeee
\node at (1,6) {${\tiny \tfrac{1}{3}}$};
\end{scope}
\draw[dash pattern=on 1pt off 1pt, opacity = 0.8] (4.25,1.2) -- (4.25,2.3);
\draw[dash pattern=on 1pt off 1pt, opacity = 0.8] (6.1,1.2) -- (6.1,2.3);
\draw[dash pattern=on 1pt off 1pt, opacity = 0.8] (4.25,2.3) -- (6.1,2.3);
\begin{scope}[shift = {(8,0)}, scale = 0.5]
\limiteeeeeeeeeeeee 
\node at (1,6) {${\tiny 0}$};
\draw[thick, fill = black, opacity = 0.4] (2-0.3,4-0.3) rectangle (2+0.3,4+0.3);
\end{scope}
\node at (7.2,1) {$\cdots$};
\node at (7.2,3) {$\cdots$};
\draw[dash pattern=on 1pt off 1pt, opacity = 0.8] (6.25,0.8) -- (6.25,2.1);
\draw[dash pattern=on 1pt off 1pt, opacity = 0.8] (6.25,2.1) -- (7,2.1);

\end{scope}

\end{tikzpicture}
	
	\caption{The construction in each of the four cases defined above. The black dots and white dots represent $x_{\alpha-1}$ and $y_{\alpha-1}$ respectively. The dashed lines represent the identified points and the boxes the new $x_{\alpha}$ and $y_{\alpha}$.}
	\label{figure_lachupadelmate}
\end{figure}

\begin{theorem}
	\label{teo_onichan} For every countable ordinal $\alpha$ the CPE class $\alpha$ is non-empty.
\end{theorem}

\begin{proof}
	By transfinite induction we will show that $(X_{\alpha},T_{\alpha})$ is in the
	CPE class $\alpha.$ We will use the following induction hypothesis on all $\beta < \alpha$.
	
	\begin{enumerate}
		\item[(a)]$\EP_{\beta}(X_{\beta},T_{\beta})=X_{\beta}^{2}.$
		\item[(b)] If $\beta$ is successor, we have that $(x_{\beta},y_{\beta})\notin\EP_{\beta-1}(X_{\beta},T_{\beta})$.
		\item[(c)] If $\beta$ is a countable limit ordinal, for every $n > 0$, we have  $(x_{\beta(0)},x_{\beta(n+1)})\notin\EP_{\beta(n)}(X_{\beta},T_{\beta})$.
%		\item[(b)] If $\beta$ is successor, for every $\beta^{\prime}<\beta$ we have that $(x,y)\notin\EP_{\beta^{\prime}}(X_{\beta},T_{\beta})$ for every $(x,y)\in U_{x_{\beta}}\times U_{y_{\beta}}$.
%		\item[(c)] If $\beta$ is limit, for every $n > 0$,  $(x,y)\notin\EP_{\beta(n)}(X_{\beta},T_{\beta})$ for every $(x,y)\in U_{x_{\beta(0)}}\times U_{x_{\beta(n+1)}}$.
	\end{enumerate}

	Clearly, showing these hypothesis is enough to prove the theorem. The base cases $\EP_{1}(X_{1},T_{1})$ and $\EP_{2}(X_{2},T_{2})$ are direct and thus there are only four cases to consider. We bring to the attention of the reader that in order to prove that (b) and (c) holds we shall proceed by contradiction. 
	\bigskip
	
	\textbf{Case 1}: $\alpha$ is an even successor ordinal.
	\bigskip
	
	(a) Let $\phi_{i}$ be the natural embedding that sends $X_{\alpha-1}$ to $X_{\alpha}$ by $\phi_i(x) = (x,i)$. This is clearly a continuous $G$-equivariant map, therefore using Lemma~\ref{lemadelencaje} and $\EP_{\alpha-1}
	(X_{\alpha-1},T_{\alpha-1})=X_{\alpha-1}^{2}$ we obtain that
	\[
(X_{\alpha-1}\times\{1\})^{2}\cup(X_{\alpha-1}\times\{2\})^{2}\cup(X_{\alpha-1}\times\{3\})^{2}\subset
	\EP_{\alpha-1}(X_{\alpha},T_{\alpha}).
	\]
	
	As $\EP_{\alpha}(X_{\alpha},T_{\alpha})$ is transitive,  $(y_{\alpha-1},1) \sim (y_{\alpha-1},2)$ and $(x_{\alpha-1},2) \sim (x_{\alpha-1},3)$ then $\EP_{\alpha}(X_{\alpha},T_{\alpha
	})=X_{\alpha}^{2}$.
	
	\bigskip
	
	(b) Suppose $((x_{\alpha-1},1),(y_{\alpha-1},3)) \in \EP_{\alpha-1}(X_{\alpha},T_{\alpha})$. As $\EP_{\alpha-1}(X_{\alpha},T_{\alpha})$ is closed we can find sequences
	$\{(x^{(n)},i^{(n)})\}_{n\in\NN} \longrightarrow (x_{\alpha-1},1)$, $\{(y^{(n)},j^{(n)})\}_{n\in\NN} \longrightarrow (y_{\alpha-1},3)$ such that $((x^{(n)},i^{(n)}),(y^{(n)},j^{(n)}))\in\EP_{\alpha-2}(X_{\alpha},T_{\alpha})$ for each $n \in \NN$. Define the projection map $\Psi_{1,3} : X_{\alpha} \to X_{\alpha-1}$ by
	$$\Psi_{1,3}(x,i) = \begin{cases}
	y_{\alpha-1} & \mbox{ if } i = 1\\
	x & \mbox{ if } i = 2\\
	x_{\alpha-1} & \mbox{ if } i = 3\\
	\end{cases}$$
	
	This map is $G$-equivariant and continuous. Using Lemma~\ref{lemadelencaje} we get $(\Psi_{1,3}(x^{(n)},i^{(n)}),\Psi_{1,3}(y^{(n)},j^{(n)}))\in\EP_{\alpha-2}(X_{\alpha-1},T_{\alpha-1})$. By definition of $\tau_{\alpha}$, there must be $n$ large enough
	such that $i^{(n)} =1$ and $j^{(n)} = 3$, therefore we conclude that $(y_{\alpha-1},x_{\alpha-1}) \in \EP_{\alpha-2}(X_{\alpha-1},T_{\alpha-1})$ thus contradicting $(b)$.
	\bigskip
	
	\textbf{Case 2}: $\alpha$ is an odd successor ordinal and $\alpha-1$ is not a
	limit ordinal.
	\bigskip
	
	(a) Using the embeddings $\phi_i(x) = (x,i)$ for each $i \in I$, we conclude using Lemma~\ref{lemadelencaje} and $\EP_{\alpha-1}(X_{\alpha-1},T_{\alpha
		-1})=X_{\alpha-1}^{2}$ that
	\[
	\bigcup_{i \in I}(X_{\alpha-1}\times\{i\})^{2}\subset\EP_{\alpha-1}(X_{\alpha
	},T_{\alpha}).
	\]
	As $\EP_{\alpha-1}(X_{\alpha},T_{\alpha})$ is transitive and for each $n \in \NN$ either $(y_{\alpha
		-1},\tfrac{1}{n})\sim(y_{\alpha-1},\tfrac{1}{n+1})$ or $(x_{\alpha
		-1},\tfrac{1}{n})\sim(x_{\alpha-1},\tfrac{1}{n+1})$ we deduce that $(X_{\alpha-1}%
	\times\{\tfrac{1}{n}\mid n>1\})^{2}\subset\EP_{\alpha-1}(X_{\alpha},T_{\alpha})$. As
	this set is dense and $\EP_{\alpha}(X_{\alpha},T_{\alpha})$ is closed we
	conclude that $\EP_{\alpha}(X_{\alpha},T_{\alpha})=X_{\alpha}^{2}$.
	\bigskip
	
	(b) Suppose $((x_{\alpha-1},1),(y_{\alpha-1},0)) \in \EP_{\alpha-1}(X_{\alpha},T_{\alpha})$. As $\EP_{\alpha-1}(X_{\alpha},T_{\alpha})$ is transitive, there is a
	finite sequence of points $\left\{ (x^{(k)}, i^{(k)}) \right\}_{k=1}^{n}$ such that $(x^{(1)},i^{(1)})%
	=(x_{\alpha -1},1),$ $(x^{(n)},i^{(n)})=(y_{\alpha -1},0)$ and $((x^{(k)}, i^{(k)}),(x^{(k+1)},i^{(k+1)}))$ is in $\EP_{\alpha-2}(X_{\alpha},T_{\alpha}).$ Let $r \in [1,n-1]$ be the largest index such that $i^{(r)}\neq 0$ and $i^{(r+1)}= 0$. Let $N$ be an odd number such that $N > \tfrac{1}{i^{(r)}}$ and define the map $\Psi_{N,N+2} : X_{\alpha} \to X_{\alpha-1}$ as
	
	$$\Psi_{N,N+2}(x,i) = \begin{cases}
	y_{\alpha-1} & \mbox{ if } i \geq \tfrac{1}{N}\\
	x & \mbox{ if } i = \tfrac{1}{N+1}\\
	x_{\alpha-1} & \mbox{ if } i \leq \tfrac{1}{N+2}\\
	\end{cases}$$
	
	As in Case 1, the map $\Psi_{N,N+2} : X_{\alpha} \to X_{\alpha-1}$ is $G$-equivariant and continuous, therefore by Lemma~\ref{lemadelencaje} we obtain that $$(\Psi_{N,N+2}(x^{(r)},i^{(r)}),\Psi_{N,N+2}(x^{(r+1)},i^{(r+1)})) = (y_{\alpha-1},x_{\alpha-1}) \in \EP_{\alpha-2}(X_{\alpha-1},T_{\alpha-1}).$$ This contradicts the induction hypothesis (b).

	\bigskip
	
	\textbf{Case 3:} $\alpha$ is a limit ordinal.
	
	\bigskip
	
	(a) By definition $\EP_{\alpha}(X_{\alpha},T_{\alpha})=\bigcup_{\beta < \alpha} \EP_{\beta}(X_{\alpha},T_{\alpha}).$ Using Lemma~\ref{lemadelencaje} with the natural embeddings we
	obtain that $\bigcup_{n \in \NN} X_{\alpha(n)}^{2}$ $\subset\EP_{\alpha}(X_{\alpha
	},T_{\alpha})$. Using that $y_{\alpha(n)}\sim y_{\alpha(n+1)}$ for
	every $n \in \NN$ and that $\EP_{\alpha}(X_{\alpha},T_{\alpha})$ is an
	equivalence relation we conclude $\EP_{\alpha}(X_{\alpha},T_{\alpha
	})=X_{\alpha}^{2}.$
	\bigskip
	
	(c) Suppose there exists $n \in \NN$ such that $(x_{\alpha(0)},x_{\alpha(n+1)})\in\EP_{\alpha(n)}(X_{\alpha},T_{\alpha}).$ We define $\phi:X_{\alpha}\rightarrow X_{\alpha(n+1)}$ as follows:
	\[
	\phi(x)=
	\begin{cases}
	x & \text{if }x\in X_{\alpha(n+1)}\\
	y_{\alpha(n+1)} & \text{otherwise}.%
	\end{cases}
	\]
	One can check that $\phi$ is $G$-equivariant and continuous and thus applying Lemma~\ref{lemadelencaje} we obtain that $\phi((x_{\alpha(0)},x_{\alpha
		(n+1)}))=(y_{\alpha(n+1)},x_{\alpha(n+1)})\in\EP_{\alpha(n)}(X_{\alpha
		(n+1)},T_{\alpha(n+1)});$ a contradiction with the induction hypothesis (b).

	\bigskip
	
	\textbf{Case 4:} $\alpha-1$ is a limit ordinal.
	
	\bigskip
	(a) The argument follows exactly as in Case 2.
	
	\bigskip
	
	(b) Suppose $((x_{\alpha(0)},1),(x_{\alpha(0)},0))\in\EP_{\alpha-1}(X_{\alpha},T_{\alpha}).$ Since $\alpha-1$ is a limit ordinal this implies there exists $n \in \NN$ such that $((x_{\alpha(0)},1),(x_{\alpha(0)},0))\in\EP_{\alpha(n)}(X_{\alpha},T_{\alpha
	}).$  Choose $m > n$ and define $\varphi_m : X_{\alpha} \to X_{\alpha-1}$ by 
	
	\[
	\varphi_m(x,i)=
	\begin{cases}
	x & \text{if } i < \tfrac{1}{m}\\
	x_{\alpha(m)} & \mbox{ otherwise. }
	\end{cases}
	\]
	
	As in the previous cases, $\varphi_m$ is $G$-equivariant. As the only identification points are of the form $(x_{\alpha(n)},\tfrac{1}{n})\sim(x_{\alpha(n)},\tfrac{1}{n+1})$ we deduce $\varphi_m$ is continuous. By Lemma~\ref{lemadelencaje} we obtain that $$(\varphi_m(x_{\alpha(0)},1),\varphi_m(x_{\alpha(0)},0)) = (x_{\alpha(m)},x_{\alpha(0)}) \in\EP_{\alpha(n)}(X_{\alpha-1},T_{\alpha-1})$$ As $m > n$ this contradicts the induction hypothesis (c).
\end{proof}
\begin{theorem}\label{teo_onichan2} For every countable ordinal $\alpha$ the asymptotic class $\alpha$ is non-empty.
\end{theorem}
\begin{proof}
	The previous construction also works in this case. We only need to check the following conditions and the rest of the proof works verbatim.
	
	\begin{enumerate}[(i)]
		\item Let $(X,T)$ and $(Y,S)$ two $G$-TDS. If $\phi:X \longrightarrow Y$ is a $G$-equivariant continuous map then for every ordinal $\alpha \geq 0$, we have
		$\phi(\R_{\alpha}(X,T))\subset\R_{\alpha}(Y,S)$.
		\item $(X_{1},T_{1}) = (\{0,1\}^G,\sigma)$ is in the asymptotic class 1.
		\item $(X_{2},T_{2})$ is in the asymptotic class 2 and $(x_2,y_2) \notin \R_1(X_2,T_2)$.
	\end{enumerate}
	
	Condition (i) is the version of Lemma~\ref{lemadelencaje} for asymptotic pairs, which holds by Remark~\ref{remark_delencaje}. Conditions (ii) and (iii) are direct. The rest of the proof is done with transfinite induction as in the previous theorem.
\end{proof}
\section{Characterization of entropy pairs for TDS with the POTP}\label{section_teofel}
We say that a set $H \subset G$ is \textbf{$P$-separated} for $P \subset G$ if
for each $d,d^{\prime}\in H$, $(P\cdot d) \cap(P\cdot d^{\prime})\neq\varnothing\implies d = d^{\prime}$.

\begin{lemma}
	\label{lem:Delone_set_condorito}  Let $G$ be a countable amenable group,
	$\{F_{n}\}_{n \in\NN}$ a F\o lner sequence and $J \subset G$ such that
	\[
	\liminf_{n \to\infty} \frac{|F_{n} \cap J|}{|F_{n}|} > 0.
	\]

	For every $P \Subset G$ there exists a subset $H \subset J$ which is
	$P$-separated and such that
	\[
	\liminf_{n \to\infty} \frac{|F_{n} \cap H|}{|F_{n}|} > 0.
	\]
	
\end{lemma}

\begin{proof}
	From Zorn's lemma, we can extract a maximal $P$-separated subset $H$ of $J$.
	Let $C \supset P^{-1}\cdot P$ be a finite subset of $G$. Every $g \in J$ must
	satisfy that $H \cap C\cdot g \neq\varnothing$, otherwise the set
	$H^{\prime}:= H \cup\{g\}$ is $P$-separated thus contradicting the maximality of $H$. Therefore,
	there is a $|C|$-to-$1$ function from $F_{n} \cap J$ to $C\cdot F_{n} \cap H$.
	This implies that,
	\[
	\frac{|C\cdot F_{n} \cap H|}{|F_{n}|} \geq\frac{|F_{n} \cap J|}{|F_{n}||C|}.
	\]
	On the other hand, as $\{F_{n}\}_{n \in\NN}$ is F\o lner, we have,
	\[
	\liminf_{n \to\infty} \frac{|C \cdot F_{n} \cap H|}{|F_{n}|} = \liminf_{n
		\to\infty} \left( \frac{|(C \cdot F_{n} \diagdown F_{n}) \cap H|}{|F_{n}|} + \frac{|F_{n} \cap H|}{|F_{n}|} \right) = 0 + \liminf_{n \to
		\infty}\frac{|F_{n} \cap H|}{|F_{n}|}.
	\]
	Therefore, we conclude that
	\[
	\liminf_{n \to\infty} \frac{|F_{n} \cap H|}{|F_{n}|} \geq\frac{1}{|C|}%
	\liminf_{n \to\infty} \frac{|F_{n} \cap J|}{|F_{n}|} > 0.
	\]
	
\end{proof}
%\begin{definition}
%Let $\mathcal{U}$ be a finite open cover of a metric space $X$. For $Y \subset X$ we denote with
%$N_{Y}(\mathcal{U)}$ the smallest cardinality of a sub-cover which covers $Y$.
%\end{definition}

\begin{theorem}
	\label{pipeteorema} Let $(X,T)$ be a $G$-TDS with the pseudo-orbit tracing property.
	
	\begin{enumerate}[(i)]
		\item If $(x,y)\in\EP(X,T)$ then $(x,y)\in\overline{\R^{\varepsilon
			}(X,T)}\diagdown\Delta$ for every $\varepsilon>0$ and there exists an
		invariant measure $\mu$ such that $x,y\in \operatorname{supp}(\mu).$
		
		\item If $(x,y)\in \R(X,T)$ and there exists an invariant
		measure $\mu$ such that $x,y\in \operatorname{supp}(\mu)$ then
		$(x,y)\in\EP(X,T)\cup\Delta$. 
	\end{enumerate}
\end{theorem}

\begin{proof}
	Let $(x,y)\in\EP(X,T).$ By Theorem~\ref{teorema_ie_es_ep} we have that $(x,y)\in$ $\IE(X,T).$ Let
	$\left\{  F_{n}\right\}_{n \in \NN}$ be a F\o lner sequence which satisfies the PET.
	
	\bigskip
	Part (i). We shall show that for any $\varepsilon>0$ there exists an $\varepsilon
	$-asymptotic pair in $B_{\varepsilon}(x)\times B_{\varepsilon}(y)$. Let
	$\delta>0$ and $S\Subset G$ be respectively the number and finite set given by
	the POTP for $\varepsilon/2$. Without loss of generality, we assume $S=S^{-1}%
	$. Let $\mathcal{C}$ be a finite $\delta$-cover
	of $X$ . For every $n\in\mathbb{N}$ we define
	\begin{align*}
	K_{n} &  :=\bigvee\nolimits_{g\in F_{n}}T^{g^{-1}}\mathcal{C}\\
	\partial K_{n} &  :=\bigvee\nolimits_{g\in S\cdot F_{n}\diagdown F_{n}%
	}T^{g^{-1}}\mathcal{C}%
	\end{align*}
	Let $(U_{x},U_{y})=(B_{\varepsilon/2}(x),B_{\varepsilon/2}(y)).$ Since $(x,y)$
	is an IE-pair, we know there exists an independence set $J\subset G$ for
	$(U_{x},U_{y})$ with positive density $D(J)$ with respect to $\{F_{n}%
	\}_{n\in\NN}$. Using that $D(J)>0$ we have that for all sufficiently large
	$n$
	%Let $Y=\left\{ x\in X:T^{g}x\in U_{x}\cup U_{y}\text{ for every }g\in
	%J\right\} .$
	\[
	{|J\cap F_{n}|}\geq{\frac{D(J)}{2}\cdot\left\vert F_{n}\right\vert }.
	\]
	We also have that for all $n\in\NN$.
	\[
	N(\partial K_{n})\leq\left\vert \mathcal{C}\right\vert ^{\left\vert S\cdot
		F_{n}\diagdown F_{n}\right\vert }.
	\]
	Since $\left\{  F_{n}\right\}  _{n\in\NN}$ is a F\o lner sequence for all
	sufficiently large $m$ we have
	\[
	2^{\frac{D(J)}{2}\left\vert F_{m}\right\vert
	}>\left\vert \mathcal{C}\right\vert ^{\left\vert S\cdot F_{m}\diagdown
		F_{m}\right\vert }\mbox{ and }{|J\cap F_{m}|}\geq{\frac{D(J)}{2}%
		\cdot\left\vert F_{m}\right\vert }.
	\]
	This implies that for large enough $m$, $2^{|J\cap F_{m}|}>N(\partial K_{m})$.
	In particular, as $J$ is an independence set for $(U_{x},U_{y})$, there exist
	$j\in J\cap F_{m}$ and $C\in\partial K_{m}$ and $x^{\prime},y^{\prime}\in C$
	such that
	\begin{align*}
	T^{j}x^{\prime} &  \in U_{x},\\
	T^{j}y^{\prime} &  \in U_{y},
	\end{align*}
	Since $x^{\prime},y^{\prime}\in C$, for every $g\in S\cdot F_{m}\diagdown
	F_{m}$ we have that $d(T^{g}x^{\prime}$,$T^{g}y^{\prime})\leq\delta$. We
	define the sequence
	\[
	\{z_{g}\}_{g\in G}\ \mbox{ such that }\ z_{g}=\left\{
	\begin{array}
	[c]{cc}%
	T^{g}x^{\prime} & \text{if }g\in F_{m}\text{ }\\
	T^{g}y^{\prime} & \mbox{otherwise.}
	\end{array}
	.\right.
	\]
	We claim that $\left\{  z_{g}\right\}  _{g\in G}$ is an $(S,\delta)$
	pseudo-orbit. \newline
	
	If $g \in F_{m}$ then either $s\cdot g\in F_{m}$ or $s\cdot g \in S\cdot F_{m}
	\setminus F_{m}$.
	
	\begin{itemize}

		\item If $s\cdot g \in F_{m}$ then $z_{s\cdot g} = T^{s\cdot g}x^{\prime} =
		T^{s}(T^{g}x^{\prime}) = T^{s}z_{g}$. 
		
		\item If $s\cdot g \in S\cdot F_{m} \setminus F_{m}$ then $z_{s\cdot g} =
		T^{s\cdot g}y^{\prime}$ and $T^{s}z_{g} =T^{s\cdot g}x^{\prime}$. As
		$d(T^{s\cdot g}x^{\prime},T^{s\cdot g}y^{\prime}) \leq\delta$ we have
		$d(z_{s\cdot g}, T^{s}z_{g}) \leq\delta$. 
	\end{itemize}
	
	If $g \notin F_{m}$ then either $s\cdot g\notin F_{m}$ or $s\cdot g \in F_{m}%
	$. 
	
	\begin{itemize}

		\item If $s\cdot g \notin F_{m}$ then $z_{s\cdot g} = T^{s\cdot g}y^{\prime}
		= T^{s}(T^{g}(y^{\prime})) = T^{s}z_{g}$. 
		
		\item If $s\cdot g \in F_{m}$, $z_{s\cdot g} = T^{s\cdot g}x^{\prime}$ and $T^{s}z_{g} =T^{s\cdot g}y^{\prime}$. As $S = S^{-1}$ we have that $g \in
		S\cdot F_{m} \setminus F_{m}$. Therefore $d(T^{s\cdot g}x^{\prime},T^{s\cdot
			g}y^{\prime}) \leq\delta$ and thus we have $d(z_{s\cdot g}, T^{s}z_{g})
		\leq\delta$. 
	\end{itemize}
	
	By the POTP there exists $z^{\prime}\in X$ that $\varepsilon/2$-traces
	$\left\{   z_{g}\right\}  _{g\in G}.$ We conclude $T^{j}z^{\prime}\in
	B_{\varepsilon}(x)$ and $T^{j}y^{\prime}\in B_{\varepsilon}(y)$. Also, for
	every $g \notin F_{m}$ we have that $d(T^{g}z^{\prime}, T^g y^{\prime})
	\leq\varepsilon/2 + d(z_{g},T^{g}y^{\prime}) = \varepsilon/2$. Therefore
	$(z^{\prime},y^{\prime})$ is $\varepsilon$-asymptotic and thus $(T^{j}%
	z^{\prime},T^j y^{\prime})$ is also $\varepsilon$-asymptotic as we intended to show.
	
	In order to construct the measure having both $x,y$ in its support, let $n
	\in\NN$ and consider an independence set $J_{n}$ for $(B_{\frac{1}{n}%
	}(x),B_{\frac{1}{n}}(y))$. By independence for any $m \in\NN$ sufficiently
	large there is $z_{m}$ such that%
	
	\[
	\min\left\lbrace \frac{|\{g \in F_{m} \mid T^{g}z_{m} \in B_{\frac{1}{n}}(x)
		\}|}{|F_{m}|},\frac{|\{g \in F_{m} \mid T^{g}z_{m} \in B_{\frac{1}{n}}(y)
		\}|}{|F_{m}|}\right\rbrace > \frac{1}{3}D(J_{n}).
	\]

	Let $\nu_{n}$ be an accumulation point in the weak-$*$ topology of the
	sequence $\{ \frac{1}{|F_{m}|} \sum_{g \in F_{m}}\delta_{T^{g}z_{m}}\}_{m
		\in\NN}$. As $\{F_{m}\}_{m \in\NN}$ is a F\o lner sequence, $\nu_{n}$ is an
	invariant probability measure that satisfies
	\[
	\nu_{n} (B_{\frac{1}{n}}(x)) > 0 \mbox{ and } \nu_{n} (B_{\frac{1}{n}}(y)) >
	0.
	\]

	Consider $\nu= \sum_{n \geq1} \frac{1}{2^{n}}\nu_{n}$. For any neighborhood
	$U_{x}$ of $x$, there is $n \in\NN$ such that $B_{\frac{1}{n}}(x) \subset
	U_{x}$ and thus $\nu(U_{x}) \geq \frac{1}{2^{n}}\nu_{n}(U_{x}) > 0$. The same argument holds
	for $y$. Therefore $x,y \in\operatorname{supp}(\nu)$.
	
	\bigskip
	Part (ii). Let $(x,y)\in\R(X,T)$ and $\mu$ an invariant measure
	such that $x,y\in\supp(\mu)$. Let $ \varepsilon>0$, we want to construct an
	independence set for $ (B_{\varepsilon}(x)$,$B_{\varepsilon}(y)).$  Let
	$\delta>0$ and $S\Subset G$ given by the POTP for $\varepsilon/2$ and without
	loss of generality, assume $S = S^{-1}$. Let $\delta^{\prime}\leq\delta$ be
	such that for any $x^{\prime},y^{\prime}\in X$ if $d(x^{\prime},y^{\prime}) <
	\delta^{\prime}$ then for any $s \in S$ $d(T^{s} x^{\prime}, T^{s} y^{\prime})
	\leq\delta/2$. Since $(x,y)\in\R(X,T)$ there exists $K \Subset G$ such that
	for $g \notin K$ $d(T^{g} x, T^{g} y) \leq\delta/2$. Without loss of generality we can assume that $S \subset K$. We  define
	
	%Let $\{F^{\prime}_m\}_{m \in \NN}$ be a F\o lner subsequence of $\{K\cdot F_n\}_{n \in \NN}$ which satisfies the PET.
	%$\ $there exists $n\in \mathbb{N}$ such that $%
	%d(T^{g}x,T^{g}y)\leq \delta $ for all $g\in S\cdot F_{n}\diagdown F_{n}.$
	%
	
	\[
	Q_{\delta^{\prime}}(x):=\left\{  z\in X:d(T^{g}x,T^{g}z) < \delta^{\prime
	}\text{ } \forall g\in S \cdot K \diagdown K \right\} .
	\]
	Note that $Q_{\delta^{\prime}}(x)$ is a neighborhood of $x$ and therefore
	$\mu(Q_{\delta^{\prime}}(x))>0.$ Since $ \left\{  F_{n}\right\} _{n \in\NN}$ satisfies
	the PET, we can apply it to the indicator function of $Q_{\delta^{\prime}}(x)$
	and infer the existence of $w \in X$ and a positive density set $J\subset G$
	such that $T^{j}w\in Q_{\delta^{\prime}}(x)$ for every $j\in J.$ Using
	Lemma~\ref{lem:Delone_set_condorito} we can extract $H \subset J$ which is $(S
	\cdot K)$-separated and has positive density in a subsequence of $\{F_{n}\}_{n
		\in\NN}$. Let $I\subset H$ be a finite set and $\phi:I\rightarrow\left\{
	x,y\right\}  $ a function. We define the sequence $\{z_{g}\}_{g \in G}$ as
	follows
	\[
	z_{g}:=\left\{
	\begin{array}
	[c]{cc}%
	T^{g\cdot i^{-1}}\phi(i) & \text{if }g \cdot i^{-1} \in S\cdot K \text{ for
		some } i\in I  \\
	T^{g}w & \text{ otherwise} 
	\end{array}
	\right.
	\]
	As $H$ is $(S\cdot K)$-separated, if $gi_{1}^{-1} \in S \cdot K$ and
	$gi_{2}^{-1} \in S \cdot K$ then $i_{1} = i_{2}$, meaning that $\{z_{g}\}_{g
		\in G}$ is well defined. We claim that $\left\{   z_{g}\right\}  _{g\in G}$ is
	an $(S,\delta)$ pseudo-orbit. Let $g \in G$ and $s \in S$:
	
	If there is $i \in I$ such that $g\cdot i^{-1} \in S\cdot K$ we have that
	$z_{g} = T^{g\cdot i^{-1}}\phi(i)$. There are two cases to consider: 
	
	\begin{itemize}

		\item if $g \cdot i^{-1} \in K$, then $s\cdot g \cdot i^{-1} \in S\cdot K$ and
		thus $z_{s \cdot g} = T^{s\cdot g\cdot i^{-1}}\phi(i)$. Therefore
		$d(T^{s}z_{g},z_{s \cdot g}) = 0$. 
		
		\item if $g \cdot i^{-1} \in S\cdot K \diagdown K$ then either $s\cdot g \cdot
		i^{-1} \in S \cdot K$ and the distance is zero or $z_{s \cdot g} = T^{s \cdot
			g}w$. In this case, we have
		\[
		d(T^{s}z_{g}, z_{s \cdot g}) = d(T^{s\cdot g \cdot i^{-1}}\phi(i), T^{s\cdot
			g}w) \leq d(T^{s\cdot g \cdot i^{-1}}\phi(i), T^{s\cdot g \cdot i^{-1}}x) +
		d(T^{s\cdot g \cdot i^{-1}}x, T^{s\cdot g}w).
		\]
		
			On one hand, we have that $d(T^{s\cdot g \cdot i^{-1}}\phi(i), T^{s\cdot g
				\cdot i^{-1}}x)\leq\delta/2$ regardless of the value of $\phi(i)$ as $s\cdot g \cdot i^{-1}
			\notin K$. On the other hand, as $T^{i} w \in Q_{\delta^{\prime}}(x)$ and $g
			\cdot i^{-1} \in S\cdot K \diagdown K$ we have that $d(T^{g\cdot i^{-1}}x,
			T^{g} w)\leq\delta^{\prime}$ and therefore $d(T^{s\cdot g\cdot i^{-1}}x,
			T^{s\cdot g} w)\leq\delta/2$. We conclude that $d(T^{s}z_{g}, z_{s \cdot
				g})\leq\delta$.
		
	\end{itemize}

	Otherwise, there is no $i \in I$ such that $g\cdot i^{-1} \in S\cdot K$ and
	thus $z_{g} = T^{g}w$. Either the same holds for $s\cdot g$ and the distance
	is zero or $s\cdot g \cdot i^{-1} \in S\cdot K$ for some $i \in I$. as $S =
	S^{-1}$, we have that $s\cdot g \cdot i^{-1} \notin K$. Therefore we can use
	the previous argument to show that $d(T^{s}z_{g}, z_{s \cdot g})\leq\delta$.
	
	By the POTP, there exists $z^{\prime}\in X$ that $\varepsilon/2$-traces
	$\left\{  z_{g}\right\}  _{g\in G}.$ Since $S =
	S^{-1}$ and $S \subset K$, we have that $1_G \in S \cdot K$ and thus $z_{i}=\phi(i)$. This implies that
	\[
	z^{\prime}\in\bigcap_{i\in I}T^{i^{-1}}B_{\varepsilon}(\phi(i)).
	\]
	%We
	We conclude $H$ is an independence set for $(B_{\varepsilon}(x)$%
	,$B_{\varepsilon}(y)).$ 
\end{proof}

\bigskip The following lemma is known (see Lemma 6.2 in \cite{ChungLi2014}).
We give a proof for completeness.

\begin{lemma}
	\label{lema_expansivo}  Let $(X,T)$ be an expansive $G$-TDS with expansivity
	constant $\varepsilon.$ Then $\R^{\varepsilon/2}(X,T) = \R(X,T).$
\end{lemma}

\begin{proof}
	Assume that there exists $(x,y)\in\R^{\varepsilon/2}(X,T)\diagdown\R(X,T).$
	This implies there exists $\delta>0$ and an infinite set $\left\{
	h_{n}\right\}  _{n\in\mathbb{N}}\subset G$ such that
	\[
	\delta\leq d(T^{h_{n}}x,T^{h_{n}}y)\text{ for all }n\in\mathbb{N}.
	\]

	Let $x^{\prime}$ and $y^{\prime}$ be accumulation points of $\left\{
	T^{h_{n}}x\right\}  _{n\in\mathbb{N}}$ and $\left\{  T^{h_{n}}y\right\}
	_{n\in\mathbb{N}}$ $\ $respectively. We must have that $x^{\prime}\neq
	y^{\prime}.$ Furthermore\ since $(x,y)\in\R^{\varepsilon/2}(X,T)$ we have
	that  for every sequence of finite sets $\left\{  K_{n}\right\}
	_{n\in\mathbb{N}}$  such that $\cup_{n\in\mathbb{N}}K_{n}=G$ and any $m\in$
	$\mathbb{N}$ for  sufficiently large $n\in\mathbb{N}$ we have
	\[
	d(T^{k\cdot h_{n}}x,T^{k\cdot h_{n}}y)\leq\varepsilon/2\text{ for every }k\in
	K_{m}.
	\]
	This implies that
	\[
	d(T^{g}x^{\prime},T^{g}y^{\prime})\leq\varepsilon/2\text{ for every }g\in
	K_{m}.
	\]
	Since this is for any $m\in$ $\mathbb{N}$ we conclude that
	\[
	d(T^{g}x^{\prime},T^{g}y^{\prime})\leq\varepsilon/2\text{ for every }g\in G,
	\]
	which contradicts expansivity.
\end{proof}

Using Lemma~\ref{lema_expansivo} we obtain the following corollaries (in
particular we recover Meyerovitch's result \cite{Meyerovitch2017}). 

\begin{corollary}
	Let $(X,T)$ be an expansive $G$-TDS with POTP and positive topological
	entropy. Then there exists a non-trivial asymptotic pair.
\end{corollary}

\begin{proof}
	By Theorem~\ref{teoremadelblansharr} any $G$-TDS with positive topological
	entropy admits an entropy pair $(x,y)$. By Theorem \ref{pipeteorema} and Lemma~\ref{lema_expansivo}
	we get that $(x,y)\in\overline{\R(X,T)}\diagdown\Delta.$
\end{proof}

We remind the reader that $\mathcal{S}$ is the set of all $x \in X$ for which there exists an invariant measure $\mu$ for which $x \in \operatorname{supp}(\mu)$ and $\Delta _{\mathcal{S}}=\left\{ (x,x)\in X^{2}: x \in \mathcal{S} \right\}$ the diagonal of $\mathcal{S}$.

\begin{corollary}\label{corolario_guajalote_mole}
	Let $(X,T)$ be an expansive $G$-TDS with POTP. Then
	\[  
	\overline{\R(X,T) \cap \mathcal{S}^2} \subset \EP(X,T)\cup \Delta \subset \overline{\R(X,T)} \cap \mathcal{S}^2\]
\end{corollary}

\begin{proof}
	By the first part of Theorem \ref{pipeteorema} and Lemma~\ref{lema_expansivo} we have that $\EP(X,T)\cup \Delta \subset \overline{\R(X,T)} \cap \mathcal{S}^2$. By the second part we obtain that $\R(X,T) \cap \mathcal{S}^2 \subset \EP(X,T)\cup \Delta$. As $\EP(X,T)\cup \Delta$ is closed we get $\overline{\R(X,T) \cap \mathcal{S}^2} \subset \EP(X,T)\cup \Delta$. 
\end{proof}

In general, if $(X,T)$ does not admit a fully supported measure the inclusions can be strict.

\begin{example}
	Let $X \subset \{ \blacktriangleleft,0,1,\blacktriangleright \}^{\ZZ}$ be the set of all configurations where the patterns $\{\blacktriangleright 0, \blacktriangleright 1, \blacktriangleright \blacktriangleleft, 0 \blacktriangleleft, 1\blacktriangleleft, 01, 10 \}$ do not appear. The subshift $(X,\sigma)$ is an expansive $\ZZ$-TDS with the POTP. It is easy to see that $\mathcal{S} = \{\blacktriangleleft^{\ZZ}, 0^{\ZZ}, 1^{\ZZ}, \blacktriangleright^{\ZZ}\}$ and thus that $\overline{\R(X,T) \cap \mathcal{S}^2} = \Delta_{\mathcal{S}}$. On the other hand we have $\mathcal{S}^2 \subset \overline{\R(X,T)}$ and thus $\overline{\R(X,T)} \cap \mathcal{S}^2 = \mathcal{S}^2$.

%	Note that for $i < j$ the set $A_{i,j} = \{x \in X \mid x_{i} = \blacktriangleleft, x_j = \blacktriangleright\}$ is finite but its orbit $\{\sigma^{n}(A_{i,j})\}_{n \in \ZZ}$ has infinitely many disjoint elements. Therefore any invariant measure $\mu$ satisfies $\mu(A_{i,j}) = 0$.
\end{example}

\begin{corollary}
	\label{corolario_delaescalerita} Let $(X,T)$ be an expansive $G$-TDS with POTP which admits a fully supported invariant measure. Then for every ordinal $\alpha > 0$ we have
	$\EP_{\alpha}(X,T)=\R_{\alpha}(X,T)$.
\end{corollary}

\begin{proof}
	If $(X,T)$ admits a fully supported invariant measure then $\mathcal{S}=X$. Using Corollary~\ref{corolario_guajalote_mole} we obtain that $\overline{\R(X,T)} = \EP(X,T)\cup \Delta$. By definition of both hierarchies we deduce that $\EP_{\alpha}(X,T)=\R_{\alpha}(X,T)$ for every ordinal $\alpha >0$.
\end{proof}

\section{The asymptotic hierarchy for subshifts}

A particularly interesting class of expansive $G$-TDSs are subshifts. Given a
finite alphabet $\mathcal{A}$ and $X\subset\mathcal{A}^{G}$ we say $X$ is a
\textbf{subshift} if it is closed under the product topology and invariant under the left shift action $\sigma$ defined by $\sigma^{g}(x)_{h}=x_{g^{-1}h}$. We denote by $L(X)$ the \textbf{language} of $X$ and for a $F \Subset G$ we
write $L_{F}(X) = \mathcal{A}^{F} \cap L(X)$ for the set of patterns appearing
in $X$ with support $F$. We say that a subshift is of \textbf{finite type (SFT)} if it is equal to the complement of the shift-closure of a finite union of cylinders. We say it is a \textbf{sofic} subshift if it is a factor of an SFT, that is, the image of an SFT under a $G$-equivariant continuous map.

In the case where $X$ is a subshift one can characterize asymptotic pairs as
follows: $(x,y)\in X^{2}$ is an asymptotic pair if and only if there exists
$F\Subset G$ such that $x|_{G\setminus F}=y|_{G\setminus F}$.

\begin{example}
	The sunny side up subshift $X_{\leq 1}$ is in the asymptotic class $0$. 
	\[
	X_{\leq1} = \{x \in\{0,1\}^{\mathbb{Z}} \mid1 \in\{x_{n},x_{m}\} \implies n =
	m\}.
	\]
\end{example}

A subshift in the asymptotic class $0$ is extremely simple. It consists of a
single asymptotic class and thus it is either finite or countable and contains a unique uniform configuration. 
In particular, the only SFT with a fully supported measure satisfying this
property is the one consisting of a unique uniform configuration. 

As the following example shows, constructing examples in the asymptotic class 1 is quite simple. 
\begin{example}
\label{example_full_shift_class_1} Let $X = \{0,1\}^{G}$ be the full
$2$-shift. The fixed points $0^{G}$ and $1^{G}$
are obviously not asymptotic, but it is easy to see that $\R(X,T)$ is dense in
$X^{2}$.
\end{example}

In what follows, we will describe the second and third level of the asymptotic hierarchy in the case where $X$ is a subshift. For that, we will
need a few definitions.

\begin{definition}
	Let $X \subset\mathcal{A}^{G}$ be a subshift, $F\Subset G$ and $p,q \in
	L_{F}(X)$. We say $p,q$ are \textbf{exchangeable} if there exists an asymptotic pair
	$(x,y)$ such that $x|_{F} = p$ and $y|_{F} = q$.
\end{definition}

\begin{proposition}\label{proposotion_trivial_clase1}
A subshift $X$ has asymptotic class at most $1$ if and only if for every $F
\Subset G$ every pair $p,q \in L_{F}(X)$ is exchangeable.
\end{proposition}

\begin{proof}
	Suppose $X$ has asymptotic class at most $1$, that is, $\overline{\R(X,\sigma)} = X^2$. Given $p,q \in L_F(X)$ choose $(x,y)\in X^2$ such that $x|_{F} = p$ and $y|_{F}=q$. By definition there is $\{(x^n,y^n)\}_{n \in \NN} \subset \R(X,\sigma)$ converging to $(x,y)$. Choosing a sufficiently large $n$ such that $x^n|_{F}=x|_{F} = p$ and $y^n|_{F}=y|_{F} = q$ gives an asymptotic pair for $p,q$, therefore they are exchangeable.
	Conversely, Let $\{F_n\}_{n\in \NN} \nearrow G$ be an increasing sequence of finite subsets of $G$ and for $(x,y) \in X^2$ let $p_n := x|_{F_n}$ and $q_n := y|_{F_n}$. As every pair of patterns is exchangeable, there is an asymptotic pair $(x^n,y^n) \in \R(X,\sigma)$ for $(p_n,q_n)$. As $\{F_n\}_{n\in \NN} \nearrow G$ we have that $(x^n,y^n)$ converges to $(x,y)$ as $n$ goes to infinity, therefore $(x,y) \in \overline{\R(X,\sigma)}$.
\end{proof}

\begin{definition}
	Let $F \Subset G$ and $p,q \in\mathcal{A}^{F}$ two patterns. We say $p,q$ are \textbf{$n$-chain exchangeable} if there exists a sequence of patterns
	$r_{0},r_{1},\dots,r_{n} \in\mathcal{A}^{F}$ such that $p = r_{0}$, $q =
	r_{n}$ and $r_{i-1},r_{i}$ are exchangeable for every $i \in\{1,\dots,n\}$. We
	say that $p,q$ are \textbf{chain exchangeable} if they are $n$-chain exchangeable for
	some $n \in\mathbb{N}$.
\end{definition}

\begin{definition}
	We say that a subshift $X$ has \textbf{chain exchangeability (CE)} if every pair of
	patterns in $L(X)$ over the same support is chain exchangeable. Furthermore, We say that
	$X$ has \textbf{bounded chain exchangeability (BCE)} if there exists a uniform constant
	$N \in\mathbb{N}$ such that every pair of patterns in $L(X)$ over the same
	support is $N$-chain exchangeable.
\end{definition}

\begin{theorem}[\cite{Pavlov2013},\cite{Pavlov2017}]\label{teorema_delronniedance}
	Let $X$ be a $\ZZ^d$-SFT admitting a fully supported invariant measure. 
	\begin{enumerate}[(i)]
		\item $X$ has chain exchangeability if and only if every non-trivial zero-dimensional factor of $X$ has positive topological entropy.
		\item If $X$ has bounded chain exchangeability then it has topological CPE.
	\end{enumerate}
\end{theorem}

In~\cite{Pavlov2017} Pavlov asked whether every $\ZZ^d$-SFT that has topological CPE must satisfy bounded chain exchangeability. Theorem~\ref{teoremabonito} answers that question negatively. 

\begin{proposition}\label{prop_roniewin}
If a $G$-subshift $X$ has bounded chain exchangeability then $\R_2(X,\sigma) = X^2$. That is, it is either in the asymptotic class $0$, $1$ or $2$.
\end{proposition}

\begin{proof}
	Suppose $X$ has BCE and let $N$ be the associated constant. Given $(x,y)\in X^2$ consider an increasing sequence $\{F_n\}_{n\in \NN}\nearrow G$ of finite subsets of $G$ and the patterns $p_n = x|_{F_n}$ and $q_n = y|_{F_n}$. By BCE, there are patterns $r^n_0,\dots r^n_N$ such that: $$p_n = r^n_0, q_n = r^n_N \mbox{ and } r_i,r_{i+1} \mbox{ are exchangable.}$$
	We can therefore find configurations $z^{(n,0)},\dots,z^{(n,N-1)}$ and $\tilde{z}^{(n,1)},\dots,\tilde{z}^{(n,N)}$ such that $(z^{(n,i)},\tilde{z}^{(n,i+1)}) \in \R(X,\sigma)$ and $z^{(n,0)}|_{F_n} = r^n_0$, $\tilde{z}^{(n,N)}|_{F_n} = r^n_N$ and for $i \in \{1,\dots,N-1\}$ $z^{(n,i)}|_{F_n} = \tilde{z}^{(n,i)}|_{F_n} = r^n_i$. By sequential compactness of $X^{2N-2}$, the sequence $\{(z^{(n,0)},\dots,z^{(n,N-1)}, \tilde{z}^{(n,1)},\dots,\tilde{z}^{(n,N)})\}_{n\in \NN}$ admits a subsequence which converges to some $(z^0,\dots,z^{N-1},\tilde{z}^1,\dots,\tilde{z}^N) \in X^{2N-1}$. By definition of the sequence and the fact that we chose $\{F_n\}_{n\in \NN}\nearrow G$, it is clear that $x = z^0$ and $y = \tilde{z}^N$ and for $i \in \{1,\dots,N-1\}$ $z^i = \tilde{z}^i$. Moreover, as $(z^{(n,i)},\tilde{z}^{(n,i+1)}) \in \R(X,\sigma)$ we have that $(z^i,\tilde{z}^{i+1}) \in \overline{\R(X,\sigma)}$. Combining these two facts we obtain that $(x,y) \in (\overline{\R(X,\sigma)})^N \subset (\overline{\R(X,\sigma)})^+$. As $(x,y)$ was arbitrary and $(\overline{\R(X,\sigma)})^+ = \R_2(X,\sigma)$ we conclude $\R_2(X,\sigma) = X^2$.
\end{proof}

Using Theorem~\ref{pipeteorema} and Proposition~\ref{prop_roniewin} we generalize Theorem~\ref{teorema_delronniedance} part (ii) for arbitrary countable amenable groups.

\begin{corollary}
	Let $X$ be a $G$-SFT admitting a fully supported measure. If $X$ has bounded chain exchangeability, then it has topological CPE.
\end{corollary}

\begin{question}
	Is there a $G$-SFT in the asymptotic class 2 which does not satisfy bounded chain exchangability?
\end{question}

The answer to the previous question is negative for topologically weakly mixing subshifts, even if they are not SFTs. Recall that a dynamical system $(X,T)$ is \textbf{topologically weakly mixing} if
$(X\times X,T\times T)$ is transitive.

\begin{proposition}
\label{proposition_weakly_mixing_top_not_bce} A topologically
weakly mixing $G$-subshift is in the asymptotic class $0,1$ or $2$ if and only if it has bounded chain exchangeability.
\end{proposition}

\begin{proof}
	Suppose $X$ is not bounded chain exchangeable, then for every $n \in \NN$ there is a support $F_n$ and patterns $(p_n,q_n)$ which are not $n$-chain exchangeable. As $X^2$ is irreducible we have that there is $g_1 \in G$ such that $[p_0]\times [q_0] \cap T^{g_1}([p_1]\times[q_1]) \neq \emptyset$. Let $g_0$ be the identity of $G$ and iterate this argument to obtain a sequence $\{g_i\}_{i \in \NN}$ such that for every $m \in \NN$ the set $U_m = \bigcap_{i \leq m}T^{g_i}([p_i]\times[q_i])$ is non-empty. Choose $(x^m,y^m)\in U_m$ and extract an accumulation point $(x,y) \in X^2$. As $X$ has class at most two, there must be $N \in \NN$ and configurations $z^0,\dots, z^N$ such that $z^0 = x$, $z^N = y$ and $(z^i,z^{i+1}) \in \overline{\R(X,\sigma)}$. By definition of convergence, we have that for each $m \in \NN$ $\sigma^{g_m^{-1}}(x)|_{F_m} = p_m$ and $\sigma^{g_m^{-1}}(y)|_{F_m} = q_m$. Let $m > N$ and $r_i := \sigma^{g_m^{-1}}(z^i)|_{F_m}$. We have that $p_m = r_0$, $q_m = r_N$ and $(r_i,r_{i+1})$ is exchangeable. Hence we have a chain of length $N < m$ for the pair $(p_m,q_m)$ which is not $m$-chain-exchangeable by assumption, thus yielding a contradiction. The other direction is given by Proposition~\ref{prop_roniewin}.
\end{proof}

\begin{theorem}[\cite{Pavlov2013}]
	There exists a $\mathbb{Z}^{2}$-SFT which admits a fully supported measure and has bounded chain exchangeability for which there are non-exchangeable patterns.
\end{theorem}

\begin{corollary}\label{corolario_elronniedance}
	There exists a $\mathbb{Z}^{2}$-SFT in the CPE class 2. 
\end{corollary}

\begin{proof}
	Using Proposition~\ref{prop_roniewin} and Proposition~\ref{proposotion_trivial_clase1} we obtain that Pavlov's example is in the asymptotic class $2$. Also, since it admits a fully supported measure Theorem~\ref{pipeteorema} yields the result.
\end{proof}

\section{A $\ZZ^3$-SFT in the CPE class $3$}
From this point forward, we will use the letters $x,y,z$ exclusively to represent coordinates in $\ZZ^3$. Before presenting our main example we construct a sofic $\mathbb{Z}^{2}$%
-subshift in the asymptotic class 3. The purpose of this construction is to
illustrate the main ideas of the proof of Theorem~\ref{teoremabonito} in a
simpler setting.

\bigskip

First we construct a $\mathbb{Z}^{2}$-SFT  whose alphabet is given by the
tiles in Figure~\ref{fig_alfabeto_felipe}. We will refer to the first tile
as a \textbf{white tile} and denote it by $\square $, the second and third as \textbf{line tiles} and the rest as \textbf{corner
tiles}. The adjacency rules are those of Wang tilings, namely, two tiles can
be adjacent to each other if and only if the colors match along their shared
border, see Figure~\ref{fig_config_felipe}. We denote this SFT as $X$.

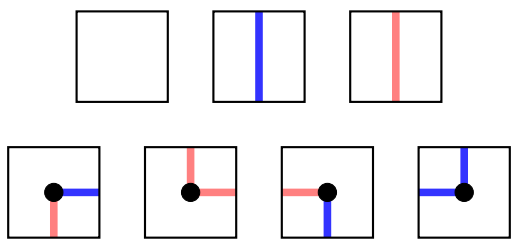
\begin{figure}[h!]
	\centering
	\begin{tikzpicture}[scale=0.6]

\def \White {
	\draw[thick] (0,0) rectangle (2,2);
	}
\def \Blue {
	\draw[line width=1mm, blue!80] (1,0) -- (1,2);
	\draw[thick] (0,0) rectangle (2,2);
	}
\def \Red {
	\draw[line width=1mm, red!50] (1,0) -- (1,2);
	\draw[thick] (0,0) rectangle (2,2);
	}
\def \RD {
	\draw[line width=1mm, red!50] (1,0) -- (1,1);
	\draw[line width=1mm, blue!80] (1,1) -- (2,1);
	\draw[fill = black] (1,1) circle (0.2);
	\draw[thick] (0,0) rectangle (2,2);
	}
\def \RU {
	\draw[line width=1mm, red!50] (1,2) -- (1,1);
	\draw[line width=1mm, red!50] (1,1) -- (2,1);
	\draw[fill = black] (1,1) circle (0.2);
	\draw[thick] (0,0) rectangle (2,2);
	}
\def \BD {
	\draw[line width=1mm, blue!80] (1,0) -- (1,1);
	\draw[line width=1mm, red!50] (1,1) -- (0,1);
	\draw[fill = black] (1,1) circle (0.2);
	\draw[thick] (0,0) rectangle (2,2);
	}
\def \BU {
	\draw[line width=1mm, blue!80] (1,2) -- (1,1);
	\draw[line width=1mm, blue!80] (1,1) -- (0,1);
	\draw[fill = black] (1,1) circle (0.2);
	\draw[thick] (0,0) rectangle (2,2);
	}

\begin{scope}[shift = {(0,0)}]
\White
\end{scope}

\begin{scope}[shift = {(3,0)}]
\Blue
\end{scope}

\begin{scope}[shift = {(6,0)}]
\Red
\end{scope}

\begin{scope}[shift = {(-1.5,-3)}]

\begin{scope}[shift = {(0,0)}]
\RD
\end{scope}

\begin{scope}[shift = {(3,0)}]
\RU
\end{scope}

\begin{scope}[shift = {(6,0)}]
\BD
\end{scope}

\begin{scope}[shift = {(9,0)}]
\BU
\end{scope}

\end{scope}
\end{tikzpicture}
	\caption{The alphabet.}
	\label{fig_alfabeto_felipe}
\end{figure}

\begin{figure}[h]
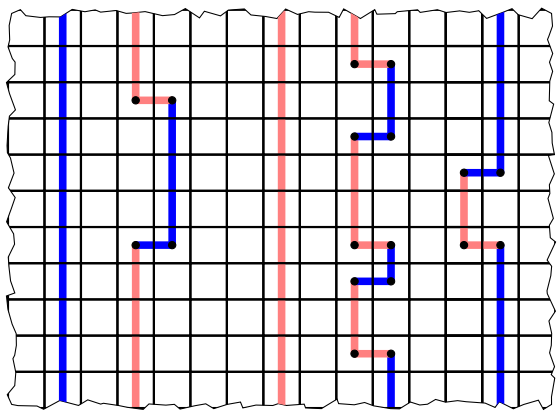

	\centering
	\include{config_felipe} 
	\caption{A configuration in $X$.}
	\label{fig_config_felipe}
\end{figure}

Let $X_{W}$ be the symbolic factor of $X$ obtained by removing the colors in
every configuration. We will call this example the \textbf{Worm Shift}. By
definition, this is a sofic $\ZZ^{2}$-subshift. See Figure~\ref{fig_config_felipe2}.

\begin{figure}[h]
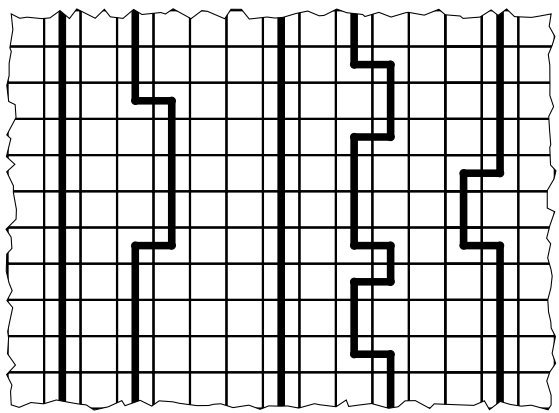

	\centering
	\include{config_felipe2} 
	\caption{A configuration in $X_{W}$.}
	\label{fig_config_felipe2}
\end{figure}

After proving a series of claims we will show that the Worm Shift is in the
asymptotic class $3$. 

\begin{definition}
	Given $c\in X_W$ we say $W\subset \mathbb{Z}^{2}$ is a \textbf{worm} if 
	
	\begin{enumerate}
		\item There exists $x\in 
		\mathbb{Z}$ such that $W\subset \left\{ x,x+1\right\} \times \mathbb{Z}$. Furthermore if $W$ is contained in ${\{x\}}\times \mathbb{Z}$ or ${\{x+1\}}\times \mathbb{Z}$ we say it is a \textbf{straight worm.}
		
		\item $W$ is infinite and 4-connected.
		
		\item $c_{(x,y)}\neq \square $  for every $(x,y)\in $ $W$.
		
		\item If $(x,y),(x+1,y)\in W$ then $(c_{(x,y)},c_{(x+1,y)}) = ( \begin{tikzpicture}[scale=0.2]\draw[line width=0.5mm, black] (1,0) -- (1,1);
		\draw[line width=0.5mm, black] (1,1) -- (2,1);
		\draw[fill = black] (1,1) circle (0.2);
		\draw[thick] (0,0) rectangle (2,2);\end{tikzpicture}, \begin{tikzpicture}[scale=0.2]\draw[line width=0.5mm, black] (1,2) -- (1,1);
		\draw[line width=0.5mm, black] (1,1) -- (0,1);
		\draw[fill = black] (1,1) circle (0.2);
		\draw[thick] (0,0) rectangle (2,2);\end{tikzpicture})$ or $(c_{(x,y)},c_{(x+1,y)}) = ( \begin{tikzpicture}[scale=0.2]\draw[line width=0.5mm, black] (1,2) -- (1,1);
		\draw[line width=0.5mm, black] (1,1) -- (2,1);
		\draw[fill = black] (1,1) circle (0.2);
		\draw[thick] (0,0) rectangle (2,2);\end{tikzpicture}, \begin{tikzpicture}[scale=0.2]\draw[line width=0.5mm, black] (1,0) -- (1,1);
		\draw[line width=0.5mm, black] (1,1) -- (0,1);
		\draw[fill = black] (1,1) circle (0.2);
		\draw[thick] (0,0) rectangle (2,2);\end{tikzpicture})$.
	\end{enumerate}
\end{definition}

Given the constraints of the SFT $X$ it is not hard to check the following
remarks. 
%\begin{remark}
%	\label{worm}  Given $c\in X_{W}$ and a worm $W$ then there exists $i\in 
%	\mathbb{Z}$ such that $W\subset \left\{ i,i+1\right\} \times \mathbb{Z}$. 
%\end{remark}

\begin{remark}
	\label{worm2}Given $c\in X_{W}$ and $c_{(x,y)}\neq \square $ there exists a
	worm $W$ such that $(x,y)\in W.$
\end{remark}

\begin{remark}
	If $c\in X_{W}$ and $(c,\cuadri)\in \R(X_{W},\sigma )$ then $c=\cuadri.$
\end{remark}

We begin by showing that $\R_{2}(X_{W},\sigma )\neq X_{W}^{2}.$  

\begin{claim}
	\label{claim_sofico_3} For each $c\in X_{W}$ with $c\neq \cuadri$ we have $%
	(c,\cuadri)\notin \R_{1}(X_{W},\sigma ).$
\end{claim}

\begin{proof}
	Assume that $(c,\cuadri)\in \R_{1}(X_{W},\sigma )$ and there exists $%
	(x,y)\in \mathbb{Z}^{2}$ such that $c_{(x,y)}\neq \square $. There exists $%
	c^{\prime },d\in X_{W}$ such that $(c^{\prime },d)\in \R(X_{W},\sigma )$
	with $c^{\prime }|_{(x,y)+[-2,2]^{2}}=c|_{(x,y)+[-2,2]^{2}}$ and $%
	d|_{(x,y)+[-2,2]^{2}}=\cuadri|_{(x,y)+[-2,2]^{2}}$. This implies that $%
	c_{(x,y)}^{\prime }\neq \square $ and hence by the previous remarks there
	exists a worm (of $c^{\prime }$) contained in $\left\{ x-1,x,x+1\right\}
	\times \mathbb{Z}.$ Since $(c^{\prime },d)\in \R(X_{W},\sigma )$ we conclude
	that there exists a worm (of $d$) contained in $\left\{ x+[-2,2]\right\}
	\times \mathbb{Z}.$ We obtain a contradiction since we already noted that $%
	d|_{(x,y)+[-2,2]^{2}}=\cuadri|_{(x,y)+[-2,2]^{2}}.$
\end{proof}

\begin{claim}
	\label{claim_sofico_1} For each $c\in X_{W}$ with $c\neq \cuadri$ we have $%
	(c,\cuadri)\notin \R_{2}(X_{W},\sigma ).$
\end{claim}

\begin{proof}
	Assume that $(c,\cuadri)\in \R_{2}(X_{W},\sigma )$. This implies there
	exists a chain $c=c^{0},\dots ,c^{n}=\cuadri$ in $X_{W}$ such that $%
	(c^{k},c^{k+1})\in \R_{1}(X_{W},\sigma )$. By Claim~\ref{claim_sofico_3}
	every element of the chain must be $\cuadri$, therefore $c=\cuadri$.
\end{proof}

It remains to show that $\R_{3}(X_{W},\sigma )=X_{W}^{2}.$ We will do it
first in simple cases and build up to a general pair using the previous
steps.

\begin{claim}
	\label{claim_sofico_2} $\R_{3}(X_{W},\sigma )=X_{W}^{2}.$
\end{claim}

\begin{proof} We proceed through five steps. The goal of the first four steps is to show that any pair of configurations having the same finite number of worms must belong to $\R_{2}(X_{W},\sigma )$. The fifth step shows that those pairs of configurations are dense.
	
	\bigskip
	
	\textbf{Step 1} Two configurations, $c$ and $d$, with one straight worm
	each, in consecutive positions.
	
		\bigskip
	
	That is, there exists $n\in \mathbb{Z}$ such
	that  $c_{(x,y)}=\square $ if and only if $x\neq n$ and $d_{(x,y)}=\square $
	if and only if $x\neq n+1$.
	Let $m \geq 1$ and $H:=\left\{  (x,y)\in\mathbb{Z}^{2}:\left\vert
	y\right\vert >m,x\neq n\right\}  \cup\left\{  (x,y)\in\mathbb{Z}%
	^{2}:\left\vert y\right\vert =m,x\neq n,x\neq n+1\right\}  \cup\left\{
	(x,y)\in\mathbb{Z}^{2}:\left\vert y\right\vert <m,x\neq n+1\right\}  .$ We
	define $e^{(m)}\in X_{W}$ as the only configuration that satisfies $e_{(x,y)}%
	^{(m)}=\square $ if and only if $(x,y)\in H$. Note that $(e^{(m)},c)\in \R(X_{W},\sigma)$
	and $e^{(m)}$ converges to $d$. This implies that $(c,d)\in \R_{1}(X_{W},\sigma ).$
	
	\bigskip
	
	\textbf{Step 2} Two configurations $c$ and $d$, with one straight worm each.
	
	\bigskip
	
	In this case, there exists $n,n^{\prime }\in \mathbb{Z}$ such that $c_{(x,y)}=\square 
	$ if and only if $x\neq n$ and $d_{(x,y)}=\square $ if and only if $x\neq
	n^{\prime }.$ Without loss of generality assume that $n<n^{\prime }.$ For $k\in \left[
	n,n^{\prime }\right] $ we define $c^{k}$ as: $c^{k}_{(x,y)}=\square $ if and
	only if $x\neq k$. Using step 1 we have that $(c^{i},c^{i+1})\in \R%
	_{1}(X_{W},\sigma ).$ Since $c^{n}=c$ and $c^{n^{\prime }}=d$ we conclude
	that $(c,d)\in \R_{2}(X_{W},\sigma ).$
	
		\bigskip
	\textbf{Step 3} Two configurations $c$ and $d,$ with one worm each (not
	necessarily straight).
		\bigskip
	
	Let $W_{c}$ and $W_{d}$ be the worms of $c$ and $d$ respectively. There
	exists $i,j\in \ZZ$ such that $W_{c}\subset \left\{ i,i+1\right\} \times 
	\mathbb{Z}$ and $W_{d}\subset \left\{ j,j+1\right\} \times \mathbb{Z}.$ We
	define $c^{\prime },d^{\prime }\in X_{W}$ as follows: $c_{(x,y)}^{\prime
	}=\square $ if and only if $y\neq i$ and $d_{(x,y)}^{\prime }=\square $ if
	and only if $y\neq j$. Using a similar argument as in step 1 we obtain that $(c,c^{\prime })\in \R%
	_{1}(X_{W},\sigma )$ and $(d,d^{\prime })\in \R_{1}(X_{W},\sigma ).$ Using
	step 2 we have $(c',d') \in \R_{2}(X_{W},\sigma )$. As $\R_{2}(X_{W},\sigma )$ is an equivalence relation we
	conclude $(c,d)\in \R_{2}(X_{W},\sigma ).$
	
		\bigskip
	\textbf{Step 4} Two configurations $c$ and $d$, with exactly $m$ worms
	each. 
		\bigskip
	
	As both configurations have finitely many worms, there exists a value $n \in \ZZ$ such that any worm in either $c$ or $d$ is contained in $\{-n,\dots, n\} \times \ZZ$. Let $e$ be the configuration containing $m$ straight worms at positions $n+2, n+4, \dots, n+2m$. Using the tools in Step 2 and Step 3 we can move the rightmost worm in $c$ towards the right one step at a time until it becomes a straight worm at $n+2m$. Iterating this procedure for all worms in $c$ from right to left we obtain that $(c,e) \in \R_{2}(X_{W},\sigma )$. Similarly, $(d,e) \in \R_{2}(X_{W},\sigma )$ and thus $(c,d) \in \R_{2}(X_{W},\sigma )$.
	
	\bigskip
	
	\textbf{Step 5} Two arbitrary configurations $c$ and $d.$
	
	\bigskip
	
	Let $Y:=\left\{ (c^{\prime },d^{\prime })\in X_{W}^{2}:\text{ }c^{\prime }%
	\text{ and }d^{\prime }\text{ have the same finite number of worms}\right\} .
	$ We will see that $Y$ is dense in $X_{W}^{2}$. Let $n\in N$ and consider $%
	[-n,n]^{2}\subset \mathbb{Z}^{2}.$ Let $N_{n}(c)$ be the number of worms for $c$	contained in $[-n-1,n+1]^{2}\times \mathbb{Z}$. 
	
	If $N_{n}(c)=N_{n}(d)$ then construct $c^{\prime }$ and $d^{\prime }$ as the
	configurations with $N_{n}(c)$ worms in same positions as $c$ and $d$. Otherwise without loss of generality $N_{n}(c)<N_{n}(d)$. We construct $d^{\prime }$ exactly as before and
	construct $c^{\prime }$ exactly as before in the window $[-n-1,n+1]^{2}%
	\times \mathbb{Z}$ and then add $N_{n}(d)-N_{n}(c)$ worms outside this
	window. In both cases $(c^{\prime },d^{\prime })\in Y$ coincide with $(c,d).$ in $[-n,n]^2$. As $n$ is arbitrary this shows that $Y$ is dense in $X^2_W$.
	
	By step 4 we have $Y \subset \R_2(X_{W},\sigma
	)$. As $\R_{3}(X_{W},\sigma
	)$ is closed we conclude $(c,d) \in \R_{3}(X_{W},\sigma)$.\end{proof}

\begin{proposition}
	\label{construction_sofico} The Worm Shift is in the asymptotic class $3$.
\end{proposition}

\begin{proof}[Proof of Proposition~\protect\ref{construction_sofico}]
	By Claim~\ref{claim_sofico_1} we have $\R_{2}(X_{W},\sigma )\neq X_{W}^{2}$
	and by Claim~\ref{claim_sofico_2} we have $\R_{3}(X_{W},\sigma )=X_{W}^{2}.$
\end{proof}

The Worm Shift has dense strongly periodic configurations and therefore
admits a fully supported measure. However, $X_{W}$ is not an SFT. The next construction uses similar ideas to the previous one but does indeed yield an SFT.

\begin{theorem}\label{teorema_asintoticodelflow}
There is a $\mathbb{Z}^{3}$-SFT in the asymptotic class $3$.
\end{theorem}

We begin by defining a $\ZZ^2$-SFT $X_{\mbox{struct}}$ on the alphabet $\Sigma$ given by all the tiles that can be obtained by coloring the thick black lines of the squares seen on Figure~\ref{fig_alph} with the colors blue and red (in all subsequent figures red is portrayed with a light tone and blue in a dark tone to facilitate the lecture in grayscale). The adjacency rules are those of Wang tilings, namely, two tiles can be adjacent to each other if and only if the lines and colors match along their shared border.
	
	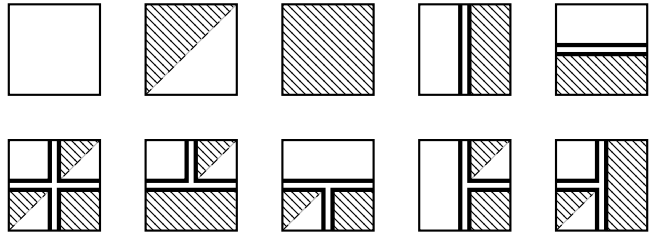
\begin{figure}[h!]
		
		\centering
		\begin{tikzpicture}[scale=0.6]

%shape of the tile
\def \e {0.1};

\begin{scope}[shift = {(0,0)}]
\draw[thick] (0,0) rectangle (2,2);
\end{scope}

\begin{scope}[shift = {(3,0)}]
\filldraw [black!30, pattern=north west lines] (0,0) -- (2,2) -- (0,2);
\draw[dashed] (0,0) -- (2,2);
\draw[thick] (0,0) rectangle (2,2);
\end{scope}

\begin{scope}[shift = {(6,0)}]
\filldraw [black!30, pattern=north west lines] (0,0) rectangle (2,2);
\draw[thick] (0,0) rectangle (2,2);
\end{scope}

\begin{scope}[shift = {(9,0)}]
\filldraw [black!30, pattern=north west lines] (1.1,0) rectangle (2,2);
\draw[ultra thick] (1-\e,0) -- (1-\e,2);
\draw[ultra thick] (1+\e,0) -- (1+\e,2);
\draw[thick] (0,0) rectangle (2,2);
\end{scope}

\begin{scope}[shift = {(12,0)}]
\filldraw [black!30, pattern=north west lines] (2,0.9) rectangle (0,0);
\draw[ultra thick] (0,1-\e) -- (2,1-\e);
\draw[ultra thick] (0,1+\e) -- (2,1+\e);
\draw[thick] (0,0) rectangle (2,2);
\end{scope}

\begin{scope}[shift = {(3,0)}]

\begin{scope}[shift = {(-3,-3)}]
\filldraw [black!30, pattern=north west lines] (1.1,1.1) -- (2,2) -- (1.1,2);
\filldraw [black!30, pattern=north west lines] (0.9,0.9) -- (0,0) -- (0,0.9);
\filldraw [black!30, pattern=north west lines] (1.1,0.9) rectangle (2,0);
\draw[ultra thick] (1-\e,0) -- (1-\e,1-\e) -- (0,1-\e);
\draw[ultra thick] (1-\e,2) -- (1-\e,1+\e) -- (0,1+\e);
\draw[ultra thick] (1+\e,0) -- (1+\e,1-\e) -- (2,1-\e);
\draw[ultra thick] (1+\e,2) -- (1+\e,1+\e) -- (2,1+\e);
\draw[dashed] (0,0) -- (0.8,0.8);
\draw[dashed] (1.2,1.2) -- (2,2);
\draw[thick] (0,0) rectangle (2,2);
\end{scope}

\begin{scope}[shift = {(0,-3)}]
\filldraw [black!30, pattern=north west lines] (1.1,1.1) -- (2,2) -- (1.1,2);
\filldraw [black!30, pattern=north west lines] (2,0.9) rectangle (0,0);
\draw[ultra thick] (1-\e,2) -- (1-\e,1+\e) -- (0,1+\e);
\draw[ultra thick] (1+\e,2) -- (1+\e,1+\e) -- (2,1+\e);
\draw[ultra thick] (0,1-\e) -- (2,1-\e);
\draw[dashed] (1.2,1.2) -- (2,2);
\draw[thick] (0,0) rectangle (2,2);
\end{scope}

\begin{scope}[shift = {(3,-3)}]
\filldraw [black!30, pattern=north west lines] (0.9,0.9) -- (0,0) -- (0,0.9);
\filldraw [black!30, pattern=north west lines] (1.1,0.9) rectangle (2,0);
\draw[ultra thick] (1-\e,0) -- (1-\e,1-\e) -- (0,1-\e);
\draw[ultra thick] (1+\e,0) -- (1+\e,1-\e) -- (2,1-\e);
\draw[ultra thick] (0,1+\e) -- (2,1+\e);
\draw[dashed] (0,0) -- (0.8,0.8);
\draw[thick] (0,0) rectangle (2,2);
\end{scope}

\begin{scope}[shift = {(6,-3)}]
\filldraw [black!30, pattern=north west lines] (1.1,1.1) -- (2,2) -- (1.1,2);
\filldraw [black!30, pattern=north west lines] (1.1,0.9) rectangle (2,0);
\draw[ultra thick] (1+\e,0) -- (1+\e,1-\e) -- (2,1-\e);
\draw[ultra thick] (1+\e,2) -- (1+\e,1+\e) -- (2,1+\e);
\draw[ultra thick] (1-\e,0) -- (1-\e,2);
\draw[dashed] (1.2,1.2) -- (2,2);
\draw[thick] (0,0) rectangle (2,2);
\end{scope}

\begin{scope}[shift = {(9,-3)}]
\filldraw [black!30, pattern=north west lines] (1.1,0) rectangle (2,2);
\filldraw [black!30, pattern=north west lines] (0.9,0.9) -- (0,0) -- (0,0.9);
\draw[ultra thick] (1-\e,0) -- (1-\e,1-\e) -- (0,1-\e);
\draw[ultra thick] (1-\e,2) -- (1-\e,1+\e) -- (0,1+\e);
\draw[ultra thick] (1+\e,0) -- (1+\e,2);
\draw[dashed] (0,0) -- (0.8,0.8);
\draw[thick] (0,0) rectangle (2,2);
\end{scope}
\end{scope}
\end{tikzpicture}
		
	\caption{The alphabet, up to coloring of the wires.}
	\label{fig_alph}
\end{figure}
The first three squares are special and are called a \textbf{white tile}, a \textbf{diagonal tile} and a \textbf{dash tile} respectively. We refer to the rest of the squares as \textbf{wire tiles}. For a wire tile $t$ we define the values $\mathsf{SW}(t),\mathsf{SE}(t),\mathsf{NW}(t),\mathsf{NE}(t)$ as $0$ if the color which is closest to the southwest, southeast, northwest and northeast corners of the square $t$ is blue and as $1$ if it is red. For example, the values associated to the following tiles are: $$ \begin{tikzpicture}[scale = 0.6]\def \e {0.1};
\node at (-1,1) {$t_1 =$};
\filldraw [black!30, pattern=north west lines] (1.1,1.1) -- (2,2) -- (1.1,2);
\filldraw [black!30, pattern=north west lines] (0.9,0.9) -- (0,0) -- (0,0.9);
\filldraw [black!30, pattern=north west lines] (1.1,0.9) rectangle (2,0);
\draw[ ultra thick, red!50] (1-\e,0) -- (1-\e,1-\e) -- (0,1-\e);
\draw[ ultra thick, blue] (1-\e,2) -- (1-\e,1+\e) -- (0,1+\e);
\draw[ ultra thick, blue] (1+\e,0) -- (1+\e,1-\e) -- (2,1-\e);
\draw[ ultra thick, red!50] (1+\e,2) -- (1+\e,1+\e) -- (2,1+\e);
\draw[dashed] (0,0) -- (0.8,0.8);
\draw[dashed] (1.2,1.2) -- (2,2);
\draw[thick] (0,0) rectangle (2,2);
\node at (3,1) {$\mbox{has}$};
\node at (6,0.5) {$\mathsf{SW}(t_1) = 1$};
\node at (6,1.5) {$\mathsf{NW}(t_1) = 0$};
\node at (10,0.5) {$\mathsf{SE}(t_1) = 0$};
\node at (10,1.5) {$\mathsf{NE}(t_1) = 1$};
\end{tikzpicture}$$
$$ \begin{tikzpicture}[scale = 0.6]\def \e {0.1};
\node at (-1,1) {$t_2 =$};
\filldraw [black!30, pattern=north west lines] (1.1,1.1) -- (2,2) -- (1.1,2);
\filldraw [black!30, pattern=north west lines] (1.1,0.9) rectangle (2,0);
\draw[ultra thick, red!50] (1+\e,0) -- (1+\e,1-\e) -- (2,1-\e);
\draw[ultra thick, blue] (1+\e,2) -- (1+\e,1+\e) -- (2,1+\e);
\draw[ultra thick, blue] (1-\e,0) -- (1-\e,2);
\draw[dashed] (1.2,1.2) -- (2,2);
\draw[thick] (0,0) rectangle (2,2);
\node at (3,1) {$\mbox{has}$};
\node at (6,0.5) {$\mathsf{SW}(t_1) = 0$};
\node at (6,1.5) {$\mathsf{NW}(t_1) = 0$};
\node at (10,0.5) {$\mathsf{SE}(t_1) = 1$};
\node at (10,1.5) {$\mathsf{NE}(t_1) = 0$};
\end{tikzpicture}$$
$$ \begin{tikzpicture}[scale = 0.6]\def \e {0.1};
\node at (-1,1) {$t_3 =$};
\filldraw [black!30, pattern=north west lines] (2,0.9) rectangle (0,0);
\draw[ultra thick, red!50] (0,1-\e) -- (2,1-\e);
\draw[ultra thick, blue] (0,1+\e) -- (2,1+\e);
\draw[thick] (0,0) rectangle (2,2);
\node at (3,1) {$\mbox{has}$};
\node at (6,0.5) {$\mathsf{SW}(t_1) = 1$};
\node at (6,1.5) {$\mathsf{NW}(t_1) = 0$};
\node at (10,0.5) {$\mathsf{SE}(t_1) = 1$};
\node at (10,1.5) {$\mathsf{NE}(t_1) = 0$};
\end{tikzpicture}$$
The global structure of a configuration in $X_{\mbox{struct}}$ is that of a partition of $\ZZ^2$ into squares with colored contours. Note that besides from the special case of configurations which do not contain wire tiles, the contour structure gives a unique way to color the associated partition into blue and red zones. On Figure~\ref{fig_pattern} we show an example of an admissible local configuration.

%The idea behind this construction is to interpret the two colors as different heights. Picture a configuration in $\{0,1\}^{\ZZ^2}$ and now draw a contour for every connected component of ones and of zeros, painting the contours of ones red and the contour of zeros blue. The structure we obtain after subdividing each connected contour into squares is the one that appears in a configuration in $X_{\mbox{struct}}$.

\begin{figure}[h!]
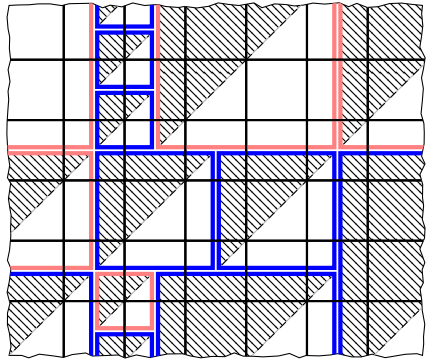

	\centering
	\include{config}
	
	\caption{A valid pattern in $X_{\mbox{struct}}$. The contours of each region determine blue and red ``zones''}
	\label{fig_pattern}
\end{figure}
	
	Before introducing our $\ZZ^3$ example, we need the following claim.
	
	\begin{claim}\label{claim.blueskyes}
		For any $n \geq 1$ and $p \in L_{[-n,n]^2}(X_{\mbox{struct}})$ there exists $c \in X_{\mbox{struct}}$ such that:
		\begin{enumerate}
			\item $c|_{[-n,n]^2} =p$.
			\item For every $u \in \ZZ^2$ such that $||u||_{\infty} > 4n$ we have $$ 
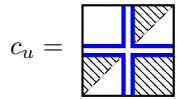
\begin{tikzpicture}[scale = 0.6]\def \e {0.1};
			\node at (-1,1) {$c_{u} =$};
			\filldraw [black!30, pattern=north west lines] (1.1,1.1) -- (2,2) -- (1.1,2);
			\filldraw [black!30, pattern=north west lines] (0.9,0.9) -- (0,0) -- (0,0.9);
			\filldraw [black!30, pattern=north west lines] (1.1,0.9) rectangle (2,0);
			\draw[ ultra thick, blue] (1-\e,0) -- (1-\e,1-\e) -- (0,1-\e);
			\draw[ ultra thick, blue] (1-\e,2) -- (1-\e,1+\e) -- (0,1+\e);
			\draw[ ultra thick, blue] (1+\e,0) -- (1+\e,1-\e) -- (2,1-\e);
			\draw[ ultra thick, blue] (1+\e,2) -- (1+\e,1+\e) -- (2,1+\e);
			\draw[dashed] (0,0) -- (0.8,0.8);
			\draw[dashed] (1.2,1.2) -- (2,2);
			\draw[thick] (0,0) rectangle (2,2);
			\end{tikzpicture}$$
		\end{enumerate}
		
	\end{claim}
	
	\begin{proof}
		Any partial square appearing in $p$ can be extended to a complete square. More precisely, if the largest length of a side of a partial square is $k$, the completed region can be constructed with a border length of at most $2k$ (the worst case given by a NW or SE corner of a square with only greys or whites inside respectively). As the length of a side of a partial square appearing in $p$ is at most $2n+1$, one can always find such a completion inside the support $[-4n, 4n]$. Let $\widetilde{p}$ be a pattern obtained after completing all of the partial squares appearing in $p$. One can choose the color of the wires bounding $\widetilde{p}$ to be only blue, and put an outgoing wire everywhere. This pattern can be extended to a $\ZZ^2$ configuration as demanded. This procedure is illustrated in Figure~\ref{fig_completion}.
	\end{proof}

	\begin{figure}[h!]
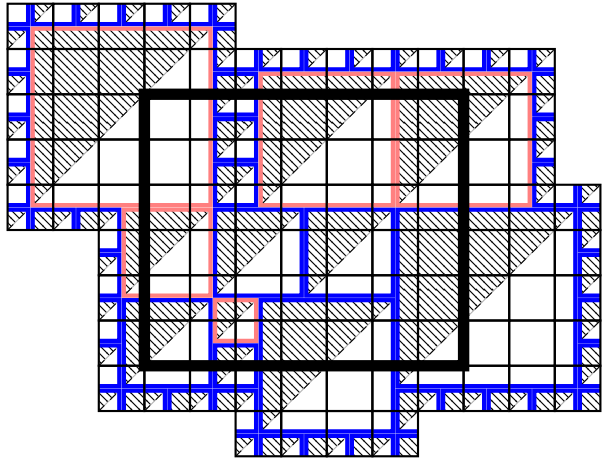

		\centering
		\include{config3}
		\caption{Completion of the pattern From Figure~\ref{fig_pattern} to a pattern which can be extended using only the blue cross wire tile.}
		\label{fig_completion}
	\end{figure}
	
%	There are three types of configurations to consider:
%	
%	\begin{enumerate}
%		\item Type 1: The all blank configuration.
%		\item Type 2: A configuration where only vertical (or horizontal) lines appear.
%		\item Type 3: At least a cross tile appears.
%	\end{enumerate}
%	
%	In a type 3 configuration the boundary of any region of white tiles contains at most one color. This is also true by vacuity in a type type 1 configuration. However, in a type 2 configuration this property may fail: for instance a configuration containing two vertical red-blue lines.
	
	 Consider now the alphabet $C = \{ \cubito, \nubito \}$ of filled cubes and empty cubes. We define the \textbf{Good Wave Shift $\GW$} as the $\ZZ^3$-SFT $\GW$ over the alphabet $\Sigma \times C$ by the following rules. 
	
	Given $c \in (\Sigma \times C)^{\ZZ^3}$ we denote by $\pi_1(c)$ and $\pi_2(c)$ its projection to the first and second coordinate respectively. For $k \in \ZZ$ we define the $z$-plane at $k$ as the set $\ZZ^2 \times \{k\} \subset \ZZ^3$. If we do not wish to specify $k$ we will just refer to a $z$-plane.
	\begin{enumerate}
		\item ($\Sigma$-Structure I) For each $k \in \ZZ$ and $c \in \GW$, $\pi_1(c)$ must contain a valid configuration of $X_{\mbox{struct}}$ in the $z$-plane at $k$.
		\item ($\Sigma$-Structure II) If a non-white tile from $\Sigma$ appears in position $(x,y,z)$, then the white tile must appear at $(x,y,z+1)$ and $(x,y,z+2)$.
		\item ($C$-Structure I) If a $\cubito$ appears at coordinate $(x,y,z)$, then a $\nubito$ must appear at $(x,y,z+1)$ and $(x,y,z+2)$.
		\item ($C$-Structure II) If a $\cubito$ appears at coordinate $(x,y,z)$, then:
		\begin{itemize}
			\item Exactly one $\cubito$ appears in either $(x+1,y,z-1)$, $(x+1,y,z)$ or $(x+1,y,z+1)$.
			\item Exactly one $\cubito$ appears in either $(x-1,y,z-1)$, $(x-1,y,z)$ or $(x-1,y,z+1)$.
			\item Exactly one $\cubito$ appears in either $(x,y+1,z-1)$, $(x,y+1,z)$ or $(x,y+1,z+1)$.
			\item Exactly one $\cubito$ appears in either $(x,y-1,z-1)$, $(x,y-1,z)$ or $(x,y-1,z+1)$.
		\end{itemize}
		\item ($\Sigma \times C$-Structure I) If a non-white tile appears at position $(x,y,z) \in \ZZ^3$, then a $\cubito$ must appear either at $(x,y,z)$ or $(x,y,z+1)$.
		\item ($\Sigma \times C$-Structure II) If a wire tile $t\in \Sigma$ appears at position $(x,y,z) \in \ZZ^3$, then:
		\begin{itemize}
			\item $\cubito$ appears at $(x,y,z+SW(t))$.
			\item $\cubito$ appears at  $(x+1,y,z+SE(t))$.
			\item $\cubito$ appears at $(x,y+1,z+NW(t))$.
			\item $\cubito$ appears at $(x+1,y+1,z+NE(t))$.
		\end{itemize}
		\item ($\Sigma \times C$-Structure III) If a $\cubito$ appears at $(x,y,z)$, then:
		\begin{itemize}
			\item If $\cubito$ appears at $(x+1,y,z+1)$, then there is a wire tile $t$ appearing at $(x,y,z)$ with $SW(t)=0$ and $SE(t) = 1$.
			\item If $\cubito$ appears at $(x+1,y,z-1)$, then there is a wire tile $t$ appearing at $(x,y,z-1)$ with $SW(t)=1$ and $SE(t) = 0$.
			\item If $\cubito$ appears at $(x,y+1,z+1)$, then there is a wire tile $t$ appearing at $(x,y,z)$ with $SW(t)=0$ and $NW(t) = 1$.
			\item If $\cubito$ appear sat $(x,y+1,z-1)$, then there is a wire tile $t$ appearing at $(x,y,z-1)$ with $SW(t)=1$ and $NW(t) = 0$.
		\end{itemize}

	\end{enumerate}
	
	Clearly all of these rules can be codified with a finite amount of forbidden patterns, therefore $\GW$ is a $\ZZ^3$-SFT. This concludes the construction, we will now prove a series of claims that are stepping stones towards Theorem~\ref{teoremabonito}.

	\begin{claim}\label{claim.struct_isolation}
		Let $c \in \GW$ and suppose $\pi_1(c)_{(x,y,z)}\neq \square$. Then for every $(i,j)\in \ZZ^2$ $\pi_1(c)_{(i,j,z+1)} = \pi_1(c)_{(i,j,z+2)} = \square$.
	\end{claim}
	
	\begin{proof}
		Let us define a map $\gamma :  X_{\mbox{struct}} \to \{\begin{tikzpicture}
		\draw (0,0) rectangle (0.25,0.25);
		\end{tikzpicture},\begin{tikzpicture}
		\filldraw [black!30] (0,0) rectangle (0.25,0.25);
		\draw (0,0) rectangle (0.25,0.25);
		\end{tikzpicture}  \}^{\ZZ^2}$ by declaring for $d \in X_{\mbox{struct}}$ that $\gamma(d)_{(x,y)} = \begin{tikzpicture}
		\draw (0,0) rectangle (0.25,0.25);
		\end{tikzpicture}$ if and only if $(x,y)$ belongs to an unbounded $4$-connected component of $d^{-1}(\begin{tikzpicture}
		\draw (0,0) rectangle (0.25,0.25);
		\end{tikzpicture})$. That is, there is an unbounded path $v \colon \NN \to \ZZ^2$ such that $v(0) = (x,y)$ and $d_{v(n)}$ is the white tile for every $n \in \NN$.	
		
		Every unbounded $4$-connected component of white tiles can be obtained as a limit of bounded right triangles in which one or more of the sides are sent to infinity. Therefore there are, up to translation, finitely many possible shapes for a maximal unbounded $4$-connected component of white tiles appearing in a configuration in $X_{\mbox{struct}}$. A section of all seven of these possibilities are shown in Figure~\ref{fig_whiteconfs2}. Therefore, the configurations of $\{\begin{tikzpicture}
		\draw (0,0) rectangle (0.25,0.25);
		\end{tikzpicture},\begin{tikzpicture}
		\filldraw [black!30] (0,0) rectangle (0.25,0.25);
		\draw (0,0) rectangle (0.25,0.25);
		\end{tikzpicture}  \}^{\ZZ^2}$ that can be obtained as images of configurations in $X_{\mbox{struct}}$ under $\gamma$ are those where the white tiles are covering disjoint and $4$-disconnected unions of the white unbounded components shown in Figure~\ref{fig_whiteconfs2}. Note that at most three unbounded components may appear in such a configuration, we illustrate an example with three unbounded components in Figure~\ref{fig_whiteconfs}.
		
		Suppose that $e,e'$ are two configurations in $\gamma(X_{\mbox{struct}})$ satisfying that for each $(x,y)\in \ZZ^2$ if $e_{(x,y)} = \begin{tikzpicture}
		\filldraw [black!30] (0,0) rectangle (0.25,0.25);
		\draw (0,0) rectangle (0.25,0.25);
		\end{tikzpicture}$ then $e'_{(x,y)} = \begin{tikzpicture}
		\draw (0,0) rectangle (0.25,0.25);
		\end{tikzpicture}$. It is not hard to prove, using the description of configurations in $\gamma(X_{\mbox{struct}})$ given above, that either $e = \cuadri$ or $e' = \cuadri$.
		
		\begin{figure}[h!]
			\centering
			\begin{tikzpicture}[scale=0.3]

\begin{scope}[shift = {(0,0)}, scale = 0.8]
\clip[draw,decorate,decoration={random steps, segment length=4pt, amplitude=1pt}] (0.2,0.2) rectangle (10.8,10.8);
\draw (0,0) grid (11,11);
\end{scope}

\begin{scope}[shift = {(10,0)}, scale = 0.8]
\clip[draw,decorate,decoration={random steps, segment length=4pt, amplitude=1pt}] (0.2,0.2) rectangle (10.8,10.8);
\filldraw [black!30] (0,0) -- (11,11) -- (0,11);
\foreach \i in {0,...,10} {\filldraw [black!30] (\i,\i) rectangle +(1,1); }
\draw (0,0) grid (11,11);
\end{scope}

\begin{scope}[shift = {(20,0)}, scale = 0.8]
\clip[draw,decorate,decoration={random steps, segment length=4pt, amplitude=1pt}] (0.2,0.2) rectangle (10.8,10.8);
\filldraw [black!30] (0,0) rectangle (11,6);
\draw (0,0) grid (11,11);
\end{scope}

\begin{scope}[shift = {(30,0)}, scale = 0.8]
\clip[draw,decorate,decoration={random steps, segment length=4pt, amplitude=1pt}] (0.2,0.2) rectangle (10.8,10.8);
\filldraw [black!30] (6,0) rectangle (11,11);
\draw (0,0) grid (11,11);
\end{scope}

\begin{scope}[shift = {(5,-10)}, scale = 0.8]
\clip[draw,decorate,decoration={random steps, segment length=4pt, amplitude=1pt}] (0.2,0.2) rectangle (10.8,10.8);
\filldraw [black!30] (-1,0) -- (10,11) -- (-1,11);
\foreach \i in {0,...,10} {\filldraw [black!30] (\i-1,\i) rectangle +(1,1); }
\filldraw [black!30] (8,0) rectangle (11,11);
\draw (0,0) grid (11,11);
\end{scope}

\begin{scope}[shift = {(15,-10)}, scale = 0.8]
\clip[draw,decorate,decoration={random steps, segment length=4pt, amplitude=1pt}] (0.2,0.2) rectangle (10.8,10.8);
\filldraw [black!30] (-1,0) -- (10,11) -- (-1,11);
\foreach \i in {0,...,10} {\filldraw [black!30] (\i-1,\i) rectangle +(1,1); }
\filldraw [black!30] (0,0) rectangle (11,3);
\draw (0,0) grid (11,11);
\end{scope}

\begin{scope}[shift = {(25,-10)}, scale = 0.8]
\clip[draw,decorate,decoration={random steps, segment length=4pt, amplitude=1pt}] (0.2,0.2) rectangle (10.8,10.8);
\filldraw [black!30] (8,0) rectangle (11,11);
\filldraw [black!30] (0,0) rectangle (11,3);
\draw (0,0) grid (11,11);
\end{scope}

\end{tikzpicture}
			
			\caption{Every unbounded configuration of white tiles in $X_{\mbox{struct}}$ must correspond to one of the above shapes up to translation.}
			\label{fig_whiteconfs2}
		\end{figure}
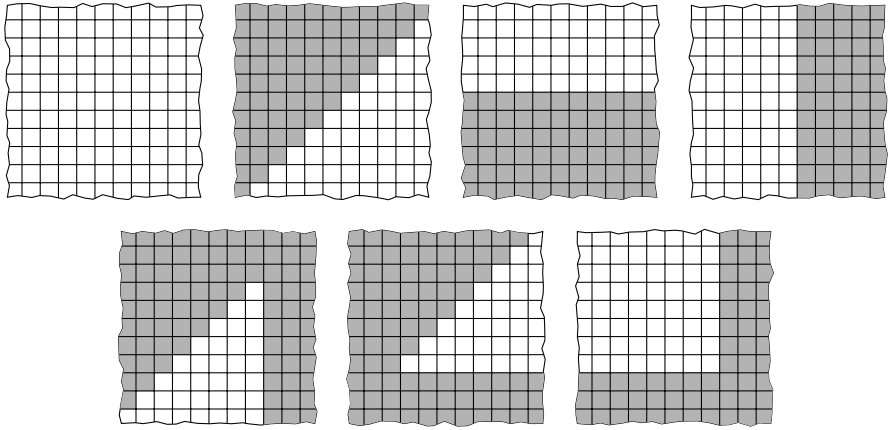
		
		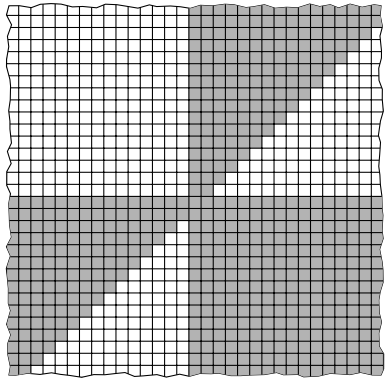
\begin{figure}[h!]
			\centering
			\begin{tikzpicture}[scale=0.2]

\begin{scope}[shift = {(0,0)}, scale = 0.8]
\clip[draw,decorate,decoration={random steps, segment length=4pt, amplitude=1pt}] (0.2,0.2) rectangle (30.8,30.8);
\foreach \i in {0,...,30} {\filldraw [black!30] (\i+1,\i) rectangle +(1,1); }
\filldraw [black!30] (15,0) rectangle  (31,15);
\filldraw [black!30] (0,-1) -- (16,15) -- (0,15);
\filldraw [black!30] (15,31) -- (15,14) -- (32,31);
\draw (0,0) grid (31,31);
\end{scope}
\end{tikzpicture}
			\caption{A possible configuration obtained through $\gamma$.}
			\label{fig_whiteconfs}
		\end{figure}
		
		By rule ($\Sigma$-Structure II) we know that any non-white tile appearing in $\pi_1(c)$ at $(x',y',z') \in \ZZ^3$ forces a white tile in $\pi_1(c)$ at $(x',y',z'+1)$ and $(x',y',z'+2)$. It is also easy to verify that, in $X_{\mbox{struct}}$, if a region is bounded by only white tiles, then it can only contain the white tiles inside the region. Putting these two facts together we obtain that not only non-white tiles force white tiles in the next two coordinates, but also white tiles appearing in bounded regions. Formally, if $\gamma(\pi_1(c)|_{\ZZ^2 \times \{z'\}})_{(x',y')} = \begin{tikzpicture}
		\filldraw [black!30] (0,0) rectangle (0.25,0.25);
		\draw (0,0) rectangle (0.25,0.25);
		\end{tikzpicture}$ then $\gamma(\pi_1(c)|_{\ZZ^2 \times \{z'+1\}})_{(x',y')} = \begin{tikzpicture}
		\draw (0,0) rectangle (0.25,0.25);
		\end{tikzpicture}$ and $\gamma(\pi_1(c)|_{\ZZ^2 \times \{z'+2\}})_{(x',y')} = \begin{tikzpicture}
		\draw (0,0) rectangle (0.25,0.25);
		\end{tikzpicture}$.
		
		Let $i \in \{0,1,2\}$, $c^i = \pi_1(c)|_{\ZZ^2 \times \{z+i\}}$, and $e^i = \gamma(c_i)$. By assumption, we have that $c^0_{(x,y)}\neq \square$ and hence $\gamma(e^0)_{(x,y)} = \begin{tikzpicture}
		\filldraw [black!30] (0,0) rectangle (0.25,0.25);
		\draw (0,0) rectangle (0.25,0.25);
		\end{tikzpicture}$. As shown before, for every $(x',y') \in \ZZ^2$ we have that if $e^0_{(x',y')} = \begin{tikzpicture}
		\filldraw [black!30] (0,0) rectangle (0.25,0.25);
		\draw (0,0) rectangle (0.25,0.25);
		\end{tikzpicture}$ then $e^1_{(x',y')} = \begin{tikzpicture}
		\draw (0,0) rectangle (0.25,0.25);
		\end{tikzpicture}$ and $e^2_{(x',y')} = \begin{tikzpicture}
		\draw (0,0) rectangle (0.25,0.25);
		\end{tikzpicture}$. From this property we conclude that $e^1 = e^2 = \cuadri$. By definition of $\gamma$ we obtain that $c^1 = c^2 = \cuadri$ as required.\end{proof}
	
%	The next two rules only affect the second layer. The ($C$-Structure I) rule simply means that no two $\cubito$ can appear adjacent to each other following the third coordinate. The ($C$-Structure II) rule means that a $\cubito$ must be extended in every direction in the $(x,y)$-plane, that is, each $\cubito$ has a north, west, south and east neighbour. However, this neighbour might be in the upper $(z+1)$ or in the lower layer $(z-1)$.
	
	\begin{definition}
		Given $c \in \GW$ a subset $W \subset \ZZ^3$ is called a \textbf{wave} if for every $w \in W$, $\pi_2(c)_w$ is $\cubito$ and there exists a function $\varphi: \ZZ^2 \to \ZZ$ such that $W = \{ (x,y,z) \in \ZZ^3 \mid z=\varphi(x,y) \}$. and for every $u,v \in \ZZ^2$ $$|\varphi(u)-\varphi(v)| \leq ||u-v||_1$$ Furthermore, a wave $W$ is a \textbf{good wave} if  $$\sup_{u,v \in \ZZ^2} |\varphi(u)-\varphi(v)| \leq 1.$$
		and a \textbf{flat wave} if $\varphi$ is constant.
	\end{definition}
	
	\begin{remark}
		Let $c \in \GW$ and assume that $W_1$ and $W_2$ are two different waves. Using rule ($C$-Structure I) we deduce that $$\inf_{w \in W_1, w' \in W_2}||w-w'||_1 > 2$$
		In particular, all the waves are disjoint.
	\end{remark}
	
	\begin{claim}\label{claim_thereisawave}
		Let $c \in \GW$, if there exists $w \in \ZZ^3$ such that $\pi_2(c)_w = \cubito$ then there exists a unique wave $W$ such that $w \in W$.
	\end{claim}
	
	\begin{proof}
		Let $w = (x,y,z)$. Let $K \subset \ZZ^2$ be a maximal set such that there is a $1$-Lipschitz function $\varphi : K \to \ZZ$ satisfying $\varphi(x,y) = z$ and $\pi_2(c)_{(x,y,\varphi(x,y))} = \cubito$. If $K \neq \ZZ^2$ pick $(i,j) \notin K$ at distance $1$ from a point $(x_1,y_1) \in K$. As $\pi_2(c)_{(x_1,y_1,\varphi(x_1,y_1))} = \cubito$, rule ($\Sigma$-Structure II) implies the existence of a $\cubito$ at exactly one of the positions $(i,j,\varphi(x_1,y_1)-1)$, $(i,j,\varphi(x_1,y_1))$ or $(i,j,\varphi(x_1,y_1)+1)$. Therefore one can extend $\varphi$ to a bigger domain contradicting the maximality of $K$. The uniqueness is direct from the previous remark.
	\end{proof}

		\begin{claim}\label{claim.twoface}
			Let $c \in \GW$. Every wave is a good wave.
		\end{claim}
		
		\begin{proof}
			Let $W$ be a wave and $\varphi : \ZZ^2 \to \ZZ$ a function defining $W$. As $\varphi$ is $1$-Lipschitz its image is an interval in $\ZZ$. If $\varphi(\ZZ^2) = \{z\}$, $\varphi(\ZZ^2) = \{z,z+1\}$ or $\varphi(\ZZ^2) = \{z-1,z\}$ then the claim is satisfied. Otherwise, there are three contiguous values $\{\eta-1,\eta,\eta+1\}$ in the image of $\varphi$. Therefore there must be two adjacent positions $u,v$ in $\ZZ^2$ which have $\varphi(u) = \eta-1$ and $\varphi(v) = \eta$. Rule ($\Sigma \times C$-Structure III) implies that a wire tile appears in the $\eta-1$ coset. On the other hand, the same argument applies to the pair $\eta,\eta+1$, showing that there is also a wire tile in the $\eta$-coset. Using Claim~\ref{claim.struct_isolation} we obtain a contradiction.
		\end{proof}

	\begin{claim}\label{claim_noclasedos}
		$\R_2(\GW,\sigma) \neq \GW^2$.
	\end{claim}
	
	\begin{proof}
		Clearly $(\square,\nubito)^{\ZZ^3} \in \GW$. Let $c \in \GW \setminus \{ (\square,\nubito)^{\ZZ^3}\}$. We claim there is a coordinate $w \in \ZZ^3$ such that $\pi_2(c)_w = \cubito$. If not, as $c \neq (\square,\nubito)^{\ZZ^3}$ there is $w' \in \ZZ^3$ such that $\pi_1(c)_{w'} \neq \square$ and thus by rule
		($\Sigma \times C$-Structure I) there is also a $\cubito$ in $\pi_2(c)$ appearing at some coordinate $w$. By Claim~\ref{claim_thereisawave} there is a wave $W$ for $c$ which is furthermore a good wave by Claim~\ref{claim.twoface}. Let $c' \in \GW$ such that $(c,c') \in \R_1(\GW,\sigma)$ and let $(\tilde{c},\tilde{c}')$ be an asymptotic pair such that $c|_{w + [-2,2]^3} = \tilde{c}|_{w + [-2,2]^3}$ and  $c'|_{w + [-2,2]^3} = \tilde{c}'|_{w + [-2,2]^3}$. As $(\tilde{c},\tilde{c}')$ are asymptotic they must coincide outside some finite region $F \Subset \ZZ^3$. Thus there must be a $\cubito$ in $\pi_2(\tilde{c}')_{\tilde{w}}$ for some $\tilde{w} \in W \cap (\ZZ \setminus F)$. Again, there is a good wave $\widetilde{W}$ for $\tilde{c}'$ passing through $\tilde{w}$ and thus $\widetilde{W} \cap (w+[-2,2]^3) \neq \emptyset$. This shows that there is a $\cubito$ appearing in $c'|_{w + [-2,2]^3}$. 
		
		We have so far shown that if $c \in \GW \setminus \{ (\square,\nubito)^{\ZZ^3}\}$  and $(c,c') \in \R_1(\GW,\sigma)$, then $c' \neq (\square,\nubito)^{\ZZ^3}$. If $(c,(\square,\nubito)^{\ZZ^3}) \in \R_2(\GW,\sigma)$, there would be a finite chain $c_0,c_1,\dots,c_n$ such that $c_0 = c$, $c_n = (\square,\nubito)^{\ZZ^3}$ and $(c_i,c_{i+1}) \in \R_1(\GW,\sigma)$. By the previous argument every $c_i$ in the chain must be $(\square,\nubito)^{\ZZ^3}$ and thus $c = (\square,\nubito)^{\ZZ^3}$ yielding a contradiction. This shows that $(c,(\square,\nubito)^{\ZZ^3}) \notin \R_2(\GW,\sigma)$.
	\end{proof}
	
%Only considering first four rules, the two coordinates of $\GW$ are completely independent from each other. On the one hand the first coordinate just consists of a sequence of configurations in $X_{\mbox{struct}}$ such that any non all-white configuration is necessarily buffered by two all-white configurations. On the other hand the second layer defines a sequence of planes which do not touch, but might change its third coordinate at a speed of $1$. The next three rules are going to force each one of these planes to remain in at most two different values of the third coordinate.

	\begin{definition}
		Let $c \in \GW$ and $W$ be a good wave. Let $\varphi$ be the function defining $W$. The \textbf{crest} of $W$ is the set $W_c = \{(x,y,\varphi(x,y)) \in W \mid  \varphi(x,y) = \max_{(i,j)\in \ZZ^2}\varphi(i,j) \}$.
	\end{definition}
	
	Let $c \in \GW$ and $W$ a good wave which is not flat. From rules ($\Sigma \times C$-Structure II) and ($\Sigma \times C$-Structure III) there is a correspondence between crests and red wires. See Figure~\ref{fig_pattern2}. On the other hand, note that flat waves correspond to planes in $X_{\mbox{struct}}$ which either do not contain wires or only contain wires of a single color.
	
	\begin{figure}[h!]
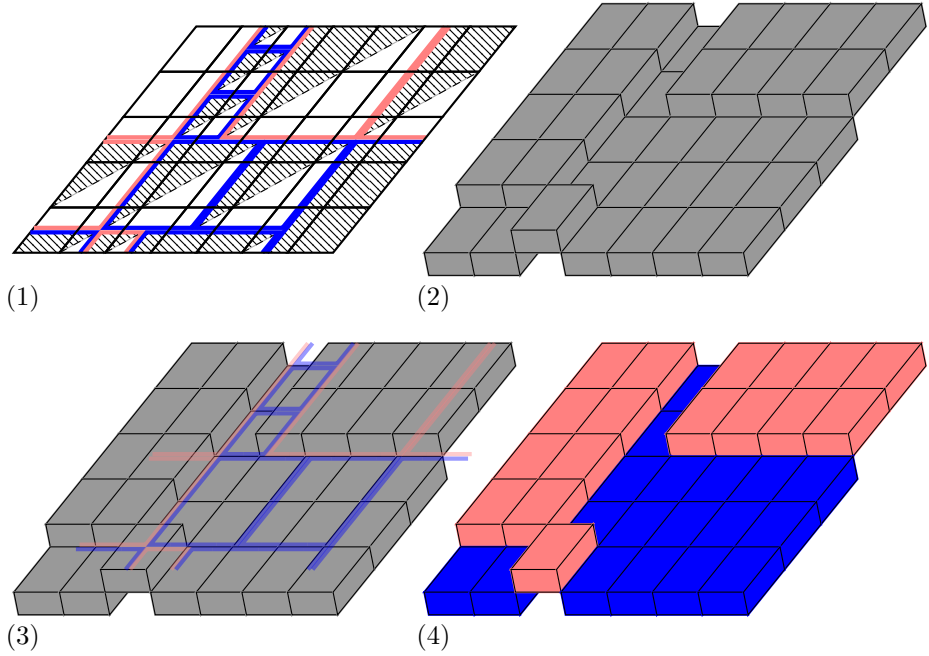

		\centering
		\include{config2}
		
		\caption{Different views of a configuration $c \in \GW$. In $(1)$ we see a configuration in $X_{\mbox{struct}}$. $(2)$ shows a good wave $W$. In $(3)$ we see the superposition of both. In $(4)$ we show the crest of the good wave in red to remark the correspondence.}
		\label{fig_pattern2}
	\end{figure}
	
	\begin{claim}\label{claim_pasting_two_configs}
		Let $c,c' \in \GW$ and $k \in \ZZ$. There is a configuration $\hat{c} \in \GW$ such that:
		
		\begin{enumerate}
			\item For every $z \in \ZZ$ such that $z < k-2$ and $(x,y)\in \ZZ^2$, $\hat{c}_{(x,y,z)} = c_{(x,y,z)}$.
			\item For every $z \in \ZZ$ such that $z > k+2$ and $(x,y)\in \ZZ^2$, $\hat{c}_{(x,y,z)} = c'_{(x,y,z)}$.
		\end{enumerate}
	\end{claim}
	
	\begin{proof}
		Let us begin by giving a sketch of how we construct $\hat{c}$. First, we  use the coordinates $k-2$ and $k+2$ to complete any good waves which might appear partially in the coordinates $k-3$ of $c$ or $k+3$ of $c'$. Second, we fill the remaining three undefined coordinates $k-1,k,k+1$ with the pair ($\square$, $\nubito$). 
		
		We shall now proceed formally. We begin by constructing a configuration $\tilde{c} \in \GW$ such that for $z < k-2$ $\tilde{c}_{(x,y,z)} = c_{(x,y,z)}$ and for $z > k-2$ $\tilde{c}_{(x,y,z)} = (\square, \nubito)$. Consider the $z$-plane at $k-3$ and the configuration $c$. If there are no wire tiles nor $\cubito$s then $\tilde{c}_{(x,y,z)}$ can be defined as $c_{(x,y,z)}$ if $z < k-2$ and ($\square$, $\nubito$) otherwise. One can check directly that $\tilde{c}$ belongs to $\GW$ and satisfies the requirement.  Otherwise, there is either (a) a coordinate $(x,y,k-3)$ where a non-white tile appears or (b) no wire tile appears in the $z$-plane at $k-3$ but a $\cubito$ appears in some coordinate of the form $(x,y,k-3)$.
		
	 In case (a), then by Claim~\ref{claim.struct_isolation} the $z$-planes at $k-2$ and $k-1$ consists uniquely of the white tile $\square$. Furthermore, the non-white tile forces a $\cubito$ to appear either at $(x,y,k-3)$ or $(x,y,k-2)$. Following the same argument as in Claim~\ref{claim.twoface} we obtain that there is a plane of $\cubito$ completely contained in the coordinates $k-3$ and $k-2$, therefore by rule ($C$-Structure I) the $z$-plane at $k-1$ contains only $\nubito$. Therefore the configuration $\tilde{c}_{(x,y,z)}$ defined as $c_{(x,y,z)}$ if $z \leq k-2$ and ($\square$, $\nubito$) otherwise belongs to $\GW$ and coincides with $c$ up to the $z$-plane at $k-1$.
	
	 In case (b) we have a $\cubito$ at a position $(x,y,k-3)$. As no wire tiles appear in the $z$-plane at $k-3$, we know there cannot be any $\cubito$ in the $z$-plane at $k-2$, therefore the plane of $\cubito$ is either completely contained in $z=k-3$, or part of it is in the $z$-plane at $k-4$. In either case, the configuration $\tilde{c}_{(x,y,z)}$ defined as $c_{(x,y,z)}$ if $z \leq k-3$ and ($\square$, $\nubito$) otherwise belongs to $\GW$ and coincides with $c$ up to the $z$-plane at $k-3$.
	
	 We have so far a configuration $\tilde{c}$ such that for $z < k-2$ $\tilde{c}_{(x,y,z)} = c_{(x,y,z)}$ and for $z > k-2$ $\tilde{c}_{(x,y,z)} = (\square, \nubito)$. Similarly, one can construct a configuration $\tilde{c}'$ such that for $z > k+2$, $\tilde{c}'_{(x,y,z)} = c'_{(x,y,z)}$ and for $z < k+2$, $\tilde{c}'_{(x,y,z)} = (\square, \nubito)$. We define $\hat{c}$ by:
	
	 $$\hat{c} = \begin{cases}
	 \tilde{c}_{(x,y,z)} \mbox{ if } z \leq k \\
	 \tilde{c}'_{(x,y,z)} \mbox{ if } z > k  \\
	 \end{cases}$$
	
	 One can verify that $\hat{c}\in \GW$ and that it satisfies the requirements of the claim.\end{proof}

	\begin{claim}\label{proposition_class_three}
		$\R_3(\GW,\sigma) = \GW^2$
	\end{claim}
	
	\begin{proof}
		The proof proceeds in four steps: in the first three steps we show that if two configurations in $\GW$ have the same finite amount of good waves, then they belong to $\R_2(\GW,\sigma)$. Then, we show that the set of all such pairs is dense in $\GW^2$.
		
		\bigskip
	
		\textbf{Step 1 } Two configurations having exactly one flat wave and no non-white tiles.
		
		\bigskip
		
		More precisely, let $k \in \ZZ$ and $c^{k} \in \GW$ be the configuration:
		$$c^k_{(x,y,z)} = \begin{cases}
		(\square, \cubito) \mbox{    if    } z = k \\
		(\square, \nubito) \mbox{    otherwise.} \\
		\end{cases}$$
		
		We will show that for any pair $k,k'$, then $(c^k,c^{k'})\in \R_2(\GW,\sigma)$. As this relation is transitive and shift-invariant, it suffices to show that $(c^0,c^{1})\in \R_1(\GW,\sigma)$.
		
		Let $n \in \NN$, $F_n = [-n,n]^3$ and consider the pattern $p^k_n = c^k|_{F_n}$. We can construct asymptotic configurations $\tilde{c}^0$ and $\tilde{c}^1$ such that $\tilde{c}^0|_{F_n} = p^0_n$ and $\tilde{c}^1|_{F_n} = p^1_n$. The construction which follows is illustrated in Figure~\ref{fig_paso1}.
		
		Indeed, we define $\tilde{c}^0$ as follows: On the first layer the configuration contains only the $\square$ tile in every $z$-coset but on $z = 0$. On $z=0$ it contains a blue square with no wire tiles on the support $[-3n+1,n+1]^2$ and the blue cross wire tile everywhere else. In the second layer it is equal to $c^0$.
		
		Define $\tilde{c}^1$ as follows: On the first layer the configuration contains only the $\square$ tile in every $z$-coset but on $z = 0$. On $z=0$ it contains a red square with no wire tiles on the support $S=[-3n+1,n+1]^2$ and the blue cross wire tile everywhere else. In the second layer it contains $\cubito$ only on the $z= 0$ and $z=1$ cosets following the first layer on $z = 0$. That is, there is a $\cubito$ on $(x,y,0)$ if and only if $(x,y) \in \ZZ^2 \setminus (S+(1,1))$ and on $(x,y,1)$ if and only if $(x,y) \in (S+(1,1))$.
		
		\begin{figure}[h!]
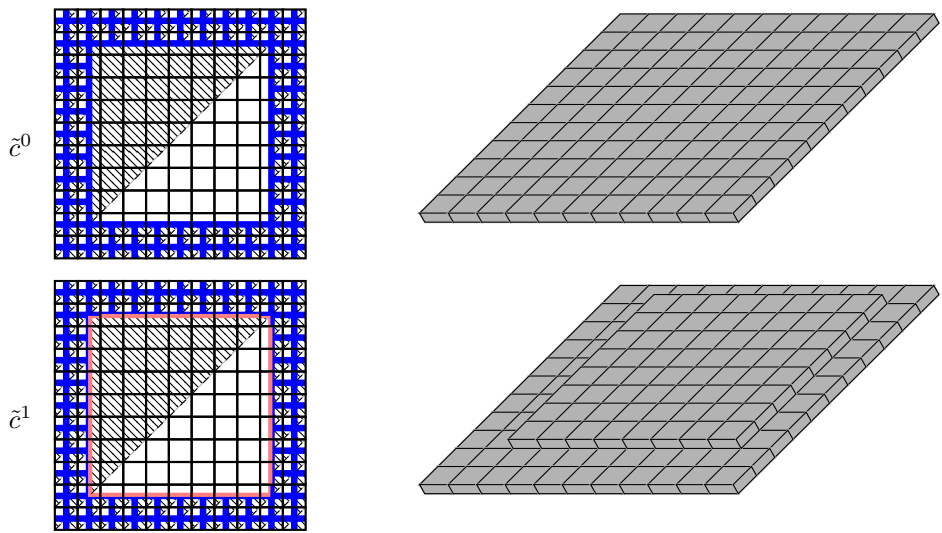

			\centering
			\include{paso_pasito_redondito_1}
			
			\caption{The asymptotic configurations $\tilde{c}^0$ (above) and $\tilde{c}^1$ (below). The only difference on the first coordinate is the color of the boundary of the inner square. Consequently, the wave at $z = 0$ is flat in $\tilde{c}^0$ but has a crest on $\tilde{c}^1$.}
			\label{fig_paso1}
		\end{figure}
		
		Clearly, we have that $\tilde{c}^0|_{\ZZ^3\setminus F_{4n}} = \tilde{c}^1|_{\ZZ^3\setminus  F_{4n}}$ and thus they are asymptotic. Also $\tilde{c}^0|_{F_{n} = p_n^0}$ and $\tilde{c}^1|_{F_{n} = p_n^1}$. As $n$ is arbitrary, we conclude $(c^0,c^1) \in \R_1(\GW,\sigma)$. Therefore $(c^0,c^{k})\in \R_2(\GW,\sigma)$.
		
		\bigskip
		
		\textbf{Step 2} Two configurations having exactly one good wave.
		
		\bigskip
		
		 Now consider a configuration $c$ which has only one good wave. If there are no non-white tiles in $c$ the wave must be flat and thus $c =c^k$ for some $k \in \ZZ$ as defined in step 1. Otherwise, there is tile which is not white at some position $(x,y,k)$. By rule ($\Sigma \times C$-Structure I) we know that there must be a $\cubito$ appearing either in $(x,y,k)$ or $(x,y,k+1)$. By Claim~\ref{claim.struct_isolation} we know that the non-white tiles are then constrained to have the same third coordinate $k$.
		
		Let $n$ be large enough such that $|k|< n$. Let $F_n = [-n,n]^3$, $H$ the $z$-plane at $k$ and consider $p = \pi_1(c)|_{F_n \cap H}$ . There are two cases to consider: (a) there is a wire tile in $p$, (b) there are no wire tiles in $p$. By Claim~\ref{claim.blueskyes} we can construct a configuration $\hat{x}\in X_{\mbox{struct}}$ such that $\hat{x}|_{[-n,n]^2} = p$ and is asymptotic to $\begin{tikzpicture}[scale = 0.15]\def \e {0.1};
		\filldraw [black!30, pattern=north west lines] (1.1,1.1) -- (2,2) -- (1.1,2);
		\filldraw [black!30, pattern=north west lines] (0.9,0.9) -- (0,0) -- (0,0.9);
		\filldraw [black!30, pattern=north west lines] (1.1,0.9) rectangle (2,0);
		\draw[ ultra thick, blue] (1-\e,0) -- (1-\e,1-\e) -- (0,1-\e);
		\draw[ ultra thick, blue] (1-\e,2) -- (1-\e,1+\e) -- (0,1+\e);
		\draw[ ultra thick, blue] (1+\e,0) -- (1+\e,1-\e) -- (2,1-\e);
		\draw[ ultra thick, blue] (1+\e,2) -- (1+\e,1+\e) -- (2,1+\e);
		\draw[dashed] (0,0) -- (0.8,0.8);
		\draw[dashed] (1.2,1.2) -- (2,2);
		\draw[thick] (0,0) rectangle (2,2);
		\end{tikzpicture}^{\infty}$. In case (a) the colors of every wire in $\hat{x}$ are determined already, whereas in (b) we can choose the color of the wires surrounding the support $[-n,n]^2$. If a $\cubito$ appears at $(0,0,k)$ we set it blue, whereas if it appears at $(0,0,k+1)$ we set it red. We can define $\hat{c} \in \GW$ by setting in the first layer the $z = k$ coset to be $\hat{x}$ and $\square$ everywhere else. $\pi_2(\hat{c})$ contains only $\cubito$s in the $z$-planes at $k$ and $k+1$ consistently with the colors of the wires in $\hat{x}$. One can check that with these choices $\hat{c}|_{F_n}=c|_{F_n}$. On the other hand, the configuration $\tilde{c}^k$ constructed in step 1 is asymptotic to $\hat{c}$. As $n$ is arbitrary  we obtain that $(c,c^k) \in \R_1(\GW,\sigma)$. Joining this with step 1 we obtain that every pair of configurations $(c,c')$ which contain a unique good wave belong to $\R_2(\GW,\sigma)$.
		
		\bigskip
		\textbf{Step 3 } Two configurations with exactly $m$ good waves.
		\bigskip
		
		Let $c$ be a configuration with exactly $m$ good waves. The idea is to think of $c$ as a superposition of $m$ disjoint good waves and modify one wave at a time to turn $c$ into a configuration with $m$ flat waves appearing at fixed positions. To achieve this we use step 2 to get rid of non-white tiles and turn the waves flat and step 1 to move flat waves to a fixed set of positions. This must be done carefully to avoid creating forbidden patterns. For instance, if we consider the configuration containing a flat wave at $z = 0$ with no non-white tiles and a flat wave at $z = 3$ with only red cross tiles on $z = 2$, then step 2 does nothing to the flat wave at $z = 0$, while it may turn the flat wave at $z=3$ into a new flat wave on $z = 2$ thus creating two $\cubito$ at distance $2$ apart and violating rule ($C$-Structure I). This possible conflict is solved by operating from the lowest coordinate towards the largest one and leaving enough space between good waves at every step.
		
		Let us proceed formally. Let $k_1 < \dots < k_m$ be the set of $k_i \in \ZZ$ such that $\pi_2(c)_{(0,0,k_i)} = \cubito$, let $N > 3m+\max(|k_1|,|k_m|)$ and let $\hat{c}$ be the configuration which has exactly $m$ flat waves at positions $-N, -N+3, \dots, -N+3(m-1)$ and no non-white tiles. We claim that it is enough to show that $(c, \hat{c}) \in \R_2(\GW,\sigma)$. Indeed, if $c_1$ and $c_2$ are two configurations which exactly $m$ waves, we can choose $N$ large enough to satisfy the above requirement for both of them and thus $(c_1,c_2) \in \R_2(\GW,\sigma)$.

		In order to show that $(c, \hat{c}) \in \R_2(\GW,\sigma)$ we construct a finite sequence of configurations $\{c_i\}$ which begins with $c$, ends with $\hat{c}$ and so that $c_{i+1}$ can be obtained from $c_i$ by applying one of the previous two steps, that is, such that $(c_i,c_{i+1}) \in \R_2(\GW,\sigma)$. More precisely, each $c_{i+1}$ is obtained from $c_i$ through a step in the following procedure:  
		
		For $j = 1,\dots,m$:
		\begin{enumerate}
			\item Use step 2 to turn the good wave at $k_j$ into a flat wave on the least possible $z$-coordinate and erase its associated non-white tiles.
			\item Iterate step 1 to move the flat wave obtained in the previous step to the coordinate $-N+3(j-1)$. 
		\end{enumerate}
		
		Clearly the configuration obtained at the end of this procedure is $\hat{c}$. Let us argue that at every step the configuration obtained from the previous one contains no forbidden patterns. 
		
		If $j = 1$, then step 2 is used on the good wave at $k_1$. On the one hand the non-white tiles are eliminated, and the $\cubito$s can only decrease their $z$-coordinate, thus no forbidden patterns can be created. The iterative application of step 1 only increases the distance between the waves and thus does not create any forbidden patterns. 
		
		If $j > 1$, we are in a configuration where the first $(j-1)$ waves are flat, have no associated non-white tiles, and are located at $-N,-N+3,\dots, -N+3(j-2)$. The application of step 2 on the $j$-th good wave again does not decrease the distance towards the $(j+1)$-th wave and by definition the distance with the $(j-1)$-th wave is at least $3$. This also holds for every application of step 1. Therefore no forbidden patterns appear.
		
%		The previous argument shows that every intermediate configuration constructed through this procedure belongs to $\GW$. The length of such a chain of configurations can be bounded by $m(2N+1)$ ($m$ good waves are to be moved, on each one step 2 is used once and step 1 at most $2N$ times) and thus does not depend on $n$. The configuration $\hat{c}$ obtained at the end of this procedure contains $m$ flat waves at positions $-N+3(j-1)$ for $j = 1,\cdots m$ and no non-white tiles. As the bound is uniform in $n$, we have $(\hat{c},c)\in \R_2(\GW,\sigma)$. As $c \in \{c_1,c_2\}$ then $(c_1,c_2)\in \R_2(\GW,\sigma)$.
		
		\bigskip
		\textbf{Step 4 } An arbitrary pair of configurations $(c_1,c_2) \in \GW^2$.
		\bigskip
		
		 We have shown so far that every pair of configurations having the same amount of good waves belongs to $\R_2(\GW,\sigma)$. It suffices to show that the pairs of configurations with the same amount of good waves are dense in $\GW^2$.
		
		Indeed, let $n \in \NN$ and use Claim~\ref{claim_pasting_two_configs} with $c_1$ and $(\square,\nubito)^{\ZZ^3}$ and $k = n$ to construct a configuration in $\GW$ which is equal to $c_1$ for all $z < n-2$ and has $(\square,\nubito)$ at any coordinate with $z > n+2$. Then use again Claim~\ref{claim_pasting_two_configs} with $(\square,\nubito)^{\ZZ^3}$ and this configuration to obtain a configuration $\hat{c}^n_1 \in \GW$ which is equal to $c_1$ for $|z| < n-2$ and has $(\square,\nubito)$ at any coordinate with $|z| > n+2$. Do the same procedure to construct $\hat{c}^n_2 \in \GW$ with the same properties but changing $c_1$ by $c_2$.
		
		Now, both $\hat{c}^n_1$ and $\hat{c}^n_2$ have a finite amount of good waves. Up to renaming, assume that $\hat{c}^n_1$ has at least as many good waves as $\hat{c}^n_2$ and define $\tilde{c}^n_1 = \hat{c}^n_1$. If $\hat{c}^n_1$ and $\hat{c}^n_2$ have the same amount of good waves, define $\tilde{c}^n_2 = \hat{c}^n_2$. Otherwise, let $r$ be the difference in the amount of good waves and define $\tilde{c}^n_2$ at $(x,y,z)$ as $\hat{c}^n_2$ if $z \leq n+2$. If $z > n+2$ we put the white tile everywhere and a $\cubito$ if and only if $z = n+2+3j$ for $j = 1,\dots,r$. That is, we add artificially $r$ flat waves to the right of $ \hat{c}^n_2$. Clearly $\tilde{c}^n_2 \in \GW$.
		
		We have that $\tilde{c}^n_1$ and $\tilde{c}^n_2$ have the same amount of good waves, hence by step 3 $(\tilde{c}^n_1,\tilde{c}^n_2) \in \R_2(\GW,\sigma)$. Moreover, by definition we have that $\tilde{c}^n_1|_{F_{n-3}} = c_1|_{F_{n-3}}$ and $\tilde{c}^n_2|_{F_{n-3}} = c_2|_{F_{n-3}}$. Therefore the sequence of pairs $\{(\tilde{c}^n_1,\tilde{c}^n_2)\}_{n \in \NN}$ converges to $(c_1,c_2)$. This shows that $(c_1,c_2) \in \R_3(\GW,\sigma)$.
	\end{proof}

	\begin{proof}[Proof of Theorem~\ref{teorema_asintoticodelflow}]
		Consider the Good Wave Shift $\GW$. Claim~\ref{claim_noclasedos} shows $\R_2(\GW,\sigma) \neq \GW^2$ and Claim~\ref{proposition_class_three} that $\R_3(\GW,\sigma) = \GW^2$. Thus $\GW$ is in the asymptotic class $3$. 
	\end{proof}
	
	In order to answer Pavlov's question, we need to show that $\GW$ has topological completely positive entropy and is not bounded chain exchangeable. This is the purpose of the following two claims.
	
		\begin{claim}\label{claim_measure}
			There is a shift-invariant fully supported measure in $\GW$.
		\end{claim}
		
		\begin{proof}
			A configuration $c$ is strongly periodic if its orbit is finite. We claim that the strongly periodic configurations are dense in $\GW$. Indeed, let $p \in L(\GW)$ be an admissible pattern with support $F_n = [-n,n]^3$. There is always a way to extend $p$ to a pattern $q \in L(\GW)$ with support $F_{4n}$ such that the boundary of the intersection with every $z$-plane consists uniquely of either:
			$$(\begin{tikzpicture}[scale = 0.15]\def \e {0.1};
			\filldraw [black!30, pattern=north west lines] (1.1,1.1) -- (2,2) -- (1.1,2);
			\filldraw [black!30, pattern=north west lines] (0.9,0.9) -- (0,0) -- (0,0.9);
			\filldraw [black!30, pattern=north west lines] (1.1,0.9) rectangle (2,0);
			\draw[ ultra thick, blue] (1-\e,0) -- (1-\e,1-\e) -- (0,1-\e);
			\draw[ ultra thick, blue] (1-\e,2) -- (1-\e,1+\e) -- (0,1+\e);
			\draw[ ultra thick, blue] (1+\e,0) -- (1+\e,1-\e) -- (2,1-\e);
			\draw[ ultra thick, blue] (1+\e,2) -- (1+\e,1+\e) -- (2,1+\e);
			\draw[dashed] (0,0) -- (0.8,0.8);
			\draw[dashed] (1.2,1.2) -- (2,2);
			\draw[thick] (0,0) rectangle (2,2);
			\end{tikzpicture},\cubito),
			(\square,\cubito) \mbox{ or } (\square, \nubito)$$
			and such that the $z$-planes at $-4n, -4n+1, 4n-1$ and $4n$ consist uniquely on $(\square, \nubito)$. Indeed, let us give a sketch of the proof. If the intersection with a $z$-plane contains a non-white tile we may use Claim~\ref{claim.blueskyes} to extend it so that the boundary consists uniquely of the pair $(\begin{tikzpicture}[scale = 0.15]\def \e {0.1};
			\filldraw [black!30, pattern=north west lines] (1.1,1.1) -- (2,2) -- (1.1,2);
			\filldraw [black!30, pattern=north west lines] (0.9,0.9) -- (0,0) -- (0,0.9);
			\filldraw [black!30, pattern=north west lines] (1.1,0.9) rectangle (2,0);
			\draw[ ultra thick, blue] (1-\e,0) -- (1-\e,1-\e) -- (0,1-\e);
			\draw[ ultra thick, blue] (1-\e,2) -- (1-\e,1+\e) -- (0,1+\e);
			\draw[ ultra thick, blue] (1+\e,0) -- (1+\e,1-\e) -- (2,1-\e);
			\draw[ ultra thick, blue] (1+\e,2) -- (1+\e,1+\e) -- (2,1+\e);
			\draw[dashed] (0,0) -- (0.8,0.8);
			\draw[dashed] (1.2,1.2) -- (2,2);
			\draw[thick] (0,0) rectangle (2,2);
			\end{tikzpicture},\cubito)$. On the remaining coordinates there can only be flat waves or no waves at all and thus it is easy to extend them so that the boundary is composed entirely either of $(\square,\cubito) \mbox{ or } (\square, \nubito)$. We use the remaining $z$-planes to complete any wave appearing partially in $p$ and fill the rest with the pair $(\square, \nubito)$. This pattern can be extended to a strongly periodic configuration $c \in \GW$ by setting $c_{(i,j,k)} := q( ((i,j,k)\mod 8n+1)-(4n,4n,4n))$. This shows that strongly periodic configurations are dense.

			Constructing a fully supported measure from a dense set of strongly periodic configurations is standard, we give a proof for completeness. Consider a dense sequence of strongly periodic configurations $\{x_n\}_{n \in \NN}$. For each configuration $x_n$ consider the uniform measure $\mu_{x_n}$ supported on the finite orbit of $x_n$. We claim that $\mu = \sum_{n > 0}2^{-n}\mu_{x_n}$ is a shift-invariant fully supported measure. Let $x\in \GW$ and pick a neighborhood $U \in \mathcal{N}_x$. As $\{x_n\}_{n \in \NN}$ is dense, there exists $N$ such that $x_{N} \in U$. Therefore $$\mu(U)\geq \frac{2^{-N}}{|\operatorname{Orb}(x_{N})|}>0.$$  And thus $x \in \supp(\mu)$. On the other hand, $\mu$ is obviously shift-invariant.\end{proof}

			\begin{claim}\label{claim_weaklymixing}
				$\GW$ is topologically weakly mixing.
			\end{claim}
			
			\begin{proof}
				Let $p,q$ be patterns in $L(\GW)$ with support $F$ and $p',q'$ patterns in $L(\GW)$ with support $F'$. Let $m$ be the largest $z$-coordinate of an element of $F$ and $k = m+3$. Choose a vector $u \in \ZZ^3$ such that the third coordinate of any element in $u+F'$ is strictly larger than $k+2$. Choose $c \in [p],\tilde{c} \in [q], c' \in \sigma^{-u}([p']),\tilde{c}' \in \sigma^{-u}([q'])$. By Claim~\ref{claim_pasting_two_configs} there is a configuration $\hat{c}$ which coincides with $c$ for $z < k-2$ and with $\tilde{c}$ for $z > k+2$. Similarly, there is $\hat{c}'$ which coincides with $c'$ for $z < k-2$ and with $\tilde{c}'$ for $z > k+2$. We clearly have that $(\tilde{c},\tilde{c}') \in ([p]\times[q]) \cap \sigma^{-u}([p']\times[q'])$. Thus showing that $\GW^2$ is transitive.
			\end{proof}
		
		\begin{theorem}\label{corolario_lapapa}
			There is a topologically weakly mixing $\mathbb{Z}^{3}$-SFT in the CPE class $3.$ 
		\end{theorem}
		
		\begin{proof}
			Consider the Good Wave Shift $\GW$. Claim~\ref{claim_weaklymixing} says that $\GW$ is topologically weakly mixing and Theorem~\ref{teorema_asintoticodelflow} says it is in the asymptotic class $3$. Furthermore Claim~\ref{claim_measure} says it admits a fully supported measure, therefore by Corollary~\ref{corolario_delaescalerita} we conclude it is in the CPE class $3$.
		\end{proof}
	
	\begin{theorem}\label{teoremabonito}
		There is a topologically weakly mixing $\mathbb{Z}^{3}$-SFT with topological CPE which does not have BCE.
	\end{theorem}
	
	\begin{proof}[Proof of Theorem~\ref{teoremabonito}]
		Consider the Good Wave Shift $\GW$. By Theorem~\ref{corolario_lapapa} $\GW$ is topologically weakly mixing and in the CPE class $3$. By Theorem~\ref{teoremadelblansharr} it must have topological CPE. On the other hand, Theorem~\ref{teorema_asintoticodelflow} says $\GW$ is in the asymptotic class $3$. Proposition~\ref{proposition_class_three} implies that it cannot have BCE.
	\end{proof}
	
As a matter of fact this yields the first known example of a transitive $\ZZ^d$-SFT with topological CPE but which does not have UPE.

\begin{question}
	Is there a $\mathbb{Z}^{2}$-SFT in the CPE class $3$?
\end{question}

\textbf{Acknowledgments}: We thank Nishant Chandgotia for very helpful remarks and prolific discussions concerning this work and Hanfeng Li for valuable comments on an earlier draft. We are also grateful to an anonymous referee for its valuable comments. We also wish to thank PIMS for their support. This research was partially supported by the ANR project CoCoGro (ANR-16-CE40-0005), NSERC (PDF-487919) and CONACyT (287764).

\bibliographystyle{alpha}
\bibliography{ref}

\begin{thebibliography}{HXY15}

\bibitem[Aki10]{akin2010general}
Ethan Akin.
\newblock {\em The general topology of dynamical systems}, volume~1.
\newblock American Mathematical Soc., 2010.

\bibitem[BHR02]{BLANCHARD2002}
F.~Blanchard, B.~Host, and S.~Ruette.
\newblock Asymptotic pairs in positive-entropy systems.
\newblock {\em Ergodic Theory and Dynamical Systems}, 22(03), jun 2002.

\bibitem[Bla92]{blanchard1992fully}
Fran{\c{c}}ois Blanchard.
\newblock Fully positive topological entropy and topological mixing.
\newblock {\em Symbolic dynamics and its applications (New Haven, CT, 1991)},
  135:95--105, 1992.

\bibitem[Bla93]{Blanchard1993}
Fran{\c{c}}ois Blanchard.
\newblock A disjointness theorem involving topological entropy.
\newblock {\em Bulletin de la Soci{\'e}t{\'e} math{\'e}matique de France},
  121(4):465--478, 1993.

\bibitem[Bow78]{Bowen1978}
Rufus Bowen.
\newblock On axiom {A} diffeomorphisms.
\newblock {\em American Mathematical Society (CBMS Regional Conference Series
  in Mathematics, 35), Providence, RI}, 1978.

\bibitem[CL14]{ChungLi2014}
Nhan-Phu Chung and Hanfeng Li.
\newblock Homoclinic groups, {IE} groups, and expansive algebraic actions.
\newblock {\em Inventiones mathematicae}, 199(3):805--858, 2014.

\bibitem[CL17]{ChungLi2017}
Nhan-Phu Chung and Keonhee Lee.
\newblock Topological stability and pseudo-orbit tracing property of group
  actions.
\newblock {\em Proceedings of the American Mathematical Society}, page~1, 2017.

\bibitem[DL11]{Downarowicz2011}
T.~Downarowicz and Y.~Lacroix.
\newblock Topological entropy zero and asymptotic pairs.
\newblock {\em Israel Journal of Mathematics}, 189(1):323--336, oct 2011.

\bibitem[GY09]{glasner2009local}
Eli Glasner and Xiangdong Ye.
\newblock Local entropy theory.
\newblock {\em Ergodic Theory and Dynamical Systems}, 29(2):321--356, 2009.

\bibitem[HXY15]{Huang2015}
Wen Huang, Leiye Xu, and Yingfei Yi.
\newblock Asymptotic pairs, stable sets and chaos in positive entropy systems.
\newblock {\em Journal of Functional Analysis}, 268(4):824--846, feb 2015.

\bibitem[KL07]{KerrLi2007}
David Kerr and Hanfeng Li.
\newblock Independence in topological and {$C^\ast$}-dynamics.
\newblock {\em Mathematische Annalen}, 338(4):869--926, 2007.

\bibitem[KL16]{KerrLiBook2016}
David Kerr and Hanfeng Li.
\newblock {\em Ergodic Theory}.
\newblock Springer International Publishing, 2016.

\bibitem[Lin01]{Lindenstrauss2001}
Elon Lindenstrauss.
\newblock Pointwise theorems for amenable groups.
\newblock {\em Inventiones Mathematicae}, 146(2):259--295, 2001.

\bibitem[LS99]{LindSchdmit1999}
Douglas Lind and Klaus Schmidt.
\newblock Homoclinic points of algebraic $\mathbb{Z}^d$-actions.
\newblock {\em Journal of the American Mathematical Society}, 12(4):953--980,
  1999.

\bibitem[Mey18]{Meyerovitch2017}
Tom Meyerovitch.
\newblock Pseudo-orbit tracing and algebraic actions of countable amenable
  groups.
\newblock {\em Ergodic Theory and Dynamical Systems}, pages 1--22, jan 2018.

\bibitem[Opr08]{Oprocha2008}
Piotr Oprocha.
\newblock Shadowing in multi-dimensional shift spaces.
\newblock {\em Colloquium Mathematicum}, 110(2):451--460, 2008.

\bibitem[Pav13]{Pavlov2013}
Ronnie Pavlov.
\newblock A characterization of topologically completely positive entropy for
  shifts of finite type.
\newblock {\em Ergodic Theory and Dynamical Systems}, 34(06):2054--2065, 2013.

\bibitem[Pav17]{Pavlov2017}
Ronnie Pavlov.
\newblock Topologically completely positive entropy and zero-dimensional
  topologically completely positive entropy.
\newblock {\em Ergodic Theory and Dynamical Systems}, pages 1--29, 2017.

\bibitem[RS61]{rohlin1961structure}
Vladimir~A. Rohlin and Yakov Sinai.
\newblock The structure and properties of invariant measurable partitions.
\newblock volume 141, pages 1038--1041. Dokl. Akad. Nauk SSSR, 1961.

\bibitem[Shu88]{shulman1988maximal}
A~Shulman.
\newblock Maximal ergodic theorems on groups.
\newblock {\em Dep. Lit. NIINTI}, 2184:114, 1988.

\bibitem[SY09]{song2009minimal}
Bailing Song and Xiangdong Ye.
\newblock A minimal completely positive entropy non-uniformly positive entropy
  example.
\newblock {\em Journal of Difference Equations and Applications}, 15(1):87--95,
  2009.

\bibitem[Wal78]{Walters1978}
Peter Walters.
\newblock On the pseudo orbit tracing property and its relationship to
  stability.
\newblock In {\em The Structure of Attractors in Dynamical Systems}, pages
  231--244. Springer Berlin Heidelberg, 1978.

\end{thebibliography}

\end{document}